\documentclass[a4paper,11pt]{amsart} 
\usepackage{amsmath,amsxtra,amssymb,latexsym, amscd,amsthm}
\usepackage[mathscr]{eucal}
\usepackage{mathrsfs}
\usepackage{bbm}
\usepackage{enumerate}
\usepackage{pict2e}
\usepackage{tikz}
\usepackage{graphicx}
\usepackage[a4paper]{geometry}
\usepackage{color}
\usepackage[colorlinks=true,%
            citecolor=blue,%
            urlcolor=darkgray]{hyperref}  
\numberwithin{equation}{section}
\theoremstyle{plain}
\newtheorem{thm}{Theorem}[section]

\newtheorem{lem}[thm]{Lemma}
\newtheorem*{euc*}{Euclidean division}
\newtheorem*{fek*}{Fekete's Lemma}
\newtheorem*{kin*}{Kingman's Subadditive Ergodic Theorem}
\newtheorem*{fur*}{Furstenberg-Kesten Theorem}
\theoremstyle{definition}

\newtheorem{rem}[thm]{Remark}

\newtheorem*{rem*}{Remark}

\newcommand{\dd}{\mathrm{d}}

\newcommand{\ii}{\mathrm{i}}
\renewcommand{\Im}{\operatorname{Im}}
\renewcommand{\Re}{\operatorname{Re}}

\newcommand{\N}{\mathbb{N}}
\newcommand{\Z}{\mathbb{Z}}

\newcommand{\R}{\mathbb{R}}
\newcommand{\C}{\mathbb{C}}

\newcommand{\T}{\mathbb{T}}

\newcommand{\cN}{\mathcal{N}}
\newcommand{\cA}{\mathcal{A}}

\newcommand{\cS}{\mathcal{S}}
\newcommand{\cB}{\mathcal{B}}

\newcommand{\cR}{\mathcal{R}}
\newcommand{\eps}{\epsilon}

\DeclareMathOperator{\supp}{supp}

\DeclareMathOperator{\varnbi}{Var_{nb}^{\mathit I}}

\DeclareMathOperator{\vari}{Var^{\mathit I}}

\DeclareSymbolFont{extraup}{U}{zavm}{m}{n}
\DeclareMathSymbol{\varheart}{\mathalpha}{extraup}{86}
\DeclareMathSymbol{\vardiamond}{\mathalpha}{extraup}{87}

\title{Quantum ergodicity for large equilateral quantum graphs}
\author{Maxime Ingremeau, Mostafa Sabri, Brian Winn}
\address{Universit\'e de Strasbourg, CNRS, IRMA UMR 7501, F-67000 Strasbourg,
France.}
\address{Laboratoiree J. A. Dieudonn\'e, UMR CNRS-UNS 7351, Universit\'e
  de Nice Sophia-Antipolis, 06108 Nice, France.}
\email{maxime.ingremeau@univ-cotedazur.fr}
\address{Department of Mathematics, Faculty of Science, Cairo University, Cairo 12613, Egypt.}
\address{Laboratoire de Math\'ematiques d'Orsay, Univ. Paris-Sud, CNRS, Universit\'e Paris-Saclay, 91405 Orsay, France.}
\email{mmsabri@sci.cu.edu.eg}
\address{Department of Mathematical Sciences, Loughborough University, Leicestershire, LE11 3TU, United Kingdom.}
\email{b.winn@lboro.ac.uk}
\usepackage{calc}
\usepackage{graphicx}

\makeatletter
\newlength{\temp@wc@width}
\newlength{\temp@wc@height}
\newcommand{\widecheck}[1]{%
  \setlength{\temp@wc@width}{\widthof{$#1$}}%
  \setlength{\temp@wc@height}{\heightof{$#1$}}%
  #1\hspace{-\temp@wc@width}%
  \raisebox{\temp@wc@height+2pt}[\heightof{$\widehat{#1}$}]%
     {\rotatebox[origin=c]{180}{\vbox to 0pt{\hbox{$\widehat{\hphantom{#1}}$}}}}%
}
\makeatother

\begin{document}

\begin{abstract}
Consider a sequence of finite regular graphs converging, in the sense of Benjamini-Schramm, to the infinite regular tree. We study the induced quantum graphs with equilateral edge lengths, Kirchhoff conditions (possibly with a non-zero coupling constant $\alpha$) and a symmetric potential $U$ on the edges. We show that in the spectral regions where the infinite quantum tree has absolutely continuous spectrum, the eigenfunctions of the converging quantum graphs satisfy a quantum ergodicity theorem. In case $\alpha=0$ and $U=0$, the limit measure is the uniform measure on the edges. In general, it has an explicit $C^1$ density. We finally prove a stronger quantum ergodicity theorem involving integral operators, the purpose of which is to study eigenfunction correlations.
\end{abstract}

\maketitle

\section{Introduction}
Quantum ergodicity, one of the fundamental theorems of quantum chaos, is a result about spatial delocalization of eigenfunctions. In its original context \cite{Shni, CdV,Zel}, it says that the eigenfunctions of the Laplace-Beltrami operator on a compact Riemannian manifold become equidistributed on the manifold in the high energy limit, provided the geodesic flow is ergodic.

In their influential work \cite{KS97}, Kottos and Smilansky suggested that the ideas and results of quantum chaos should apply to quantum graphs. It is therefore natural to inquire whether quantum ergodicity holds on quantum graphs. By a quantum graph, we mean a metric graph, equipped with a differential operator and suitable boundary conditions at each vertex. We refer the reader to Section~\ref{sec:DefQG} for a more precise definition.

Let us briefly survey some earlier results of quantum ergodicity on quantum graphs.

The first investigation appeared in \cite{BKW04}, where it was shown that the eigenfunctions of star graphs are not quantum ergodic in the high energy limit. More precisely, by studying the behaviour of the semi-classical measure when the number of bonds becomes large and the lengths approach a common value, one sees that it is not the uniform measure.

A positive result of quantum ergodicity later appeared in \cite{BKS07} for graphs associated to intervals maps. In that paper, instead of studying directly the high energy behaviour of the eigenfunctions $\psi_n$ of the quantum graph, the authors study the eigenfunctions $\phi_j(\lambda_n)$ of associated unitary operators $U(\lambda_n)$ which encode the classical evolution. It is shown that both notions of quantum ergodicity (for $\psi_n$ or $\phi_j(\lambda_n)$) are intimately related if the size of the graph goes to infinity, see also \cite{BeWi10} for a precise statement.

The previous result can serve as a motivation to consider \emph{asymptotic quantum ergodicity} for general quantum graphs. In this case, one considers a sequence of quantum graphs, first studying the high energy behaviour, then taking the size of the graph to infinity. A general strategy for studying this regime was introduced in \cite{GKP10}, and it is shown that the validity of asymptotic quantum ergodicity (AQE) depends on the spectral properties of the Markov transition matrix $M_{b,b'}(\lambda) = |U_{b,b'}(\lambda)|^2$. Namely, AQE should hold if the spectral gap of $M(\lambda)$ does not decay too fast as the graph grows larger. This can be regarded as a ``chaotic'' assumption, similar to the ergodicity of the geodesic flow on manifolds. Note that the spectral gap of quantum star graphs closes quite rapidly. 

It is natural however to study the high energy behaviour of a \emph{fixed} compact quantum graph. This question was settled later in \cite{CdV15} in the case of Kirchhoff boundary conditions, where it is shown that quantum ergodicity does not hold for a generic metric on the graph, except if it is homeomorphic to an interval or a circle. The semi-classical measures were moreover characterized, and the ``scars'' were described.

To conclude this brief survey, we mention the paper \cite{KaSc14} in which bounds on the entropy of the semi-classical measures were obtained for several families of quantum graphs. Finally, the paper \cite{BrWi16} proved that asymptotic quantum ergodicity holds for sequences of quantum graphs without back-scattering, if the underlying discrete graphs are $(q+1)$-regular expanders with few short loops. Under these ``equi-transmitting'' boundary conditions (which do not include the Kirchhoff conditions), the Markov operator $M(\lambda)$ is just $\frac{1}{q} \cB$, where $(\cB f)(b) = \sum_{b^+\in \cN_b^+} f(b^+)$ is the non-backtracking operator on the edges. It is known that if the graphs are expanders, then $\frac{1}{q}\cB$ has a uniform spectral gap \cite{ALM,A}. So this result of \cite{BrWi16} confirms the general philosophy of \cite{GKP10} in this situation. 

All the previous results were concerned with the high-energy behaviour of eigenfunctions, be it for a sequence of quantum graphs or a fixed graph. In this paper, we are interested in a different regime. Ultimately, we believe the following principle should be true~: suppose that an infinite, possibly random, quantum tree has purely absolutely continuous spectrum in an interval $I$. Consider a sequence of quantum graphs converging to this tree in the sense of Benjamini-Schramm, and suppose the underlying discrete graphs are expanders. Then as the graphs $\mathbf{G}_N$ grow large, the corresponding eigenfunctions whose energies \emph{lie in the bounded interval $I$} will satisfy a quantum ergodicity theorem.

This kind of result can be interpreted as a delocalization result for the infinite quantum tree. It says that if the tree is \emph{spectrally delocalized} (has pure AC spectrum in $I$), then it also has a form of \emph{spatial delocalization} (converging quantum graphs satisfy quantum ergodicity in $I$).

This point of view originally stemmed from the study of quantum ergodicity for discrete graphs. The first incarnations of it appeared in the proof of quantum ergodicity for the adjacency matrix of a regular graph \cite{ALM,A,BLL}. The question was later addressed in full generality in \cite{AS2}, in the framework of Schr\"odinger operators on graphs of bounded degrees. As an important application \cite{AS3}, it was shown that the eigenfunctions of the Anderson model on a large graph become equidistributed if the disorder is weak enough, in the region of AC spectrum previously provided by Klein \cite{Klein}. This was among very few theorems of delocalization for the Anderson model, the opposite regime of localization being quite well understood today. More applications appeared in \cite{AS4}, where the first results of quantum ergodicity for non-regular graphs were given. We also mention the recent paper \cite{LMS17} which studies quantum ergodicity for sequences of compact hyperbolic surfaces, also in the bounded interval regime.

The present paper is a first step towards the proof of such a general criterion for quantum graphs. We study a simple family of quantum graphs~: regular equilateral graphs, and endow them with the natural Kirchhoff conditions. We also allow for identical coupling constants $\alpha$ on the vertices and identical symmetric potentials $U$ on the edges. It is known \cite{Car97} that the corresponding infinite tree has bands of AC spectra. We confirm the general principle by showing that eigenfunctions with energies in such bands are quantum ergodic, as the graph grow large. More precisely, when the graph $\mathbf{G}_N$ is large enough, the probability measure $|\psi_j^{(N)}(x)|^2\,\dd x$ on $\mathbf{G}_N$ approaches an explicit measure $\frac{1}{N} \Psi_{\lambda_j}(x)\,\dd x$, for most eigenfunctions $\psi_j^{(N)}$. In the special case $U=\alpha=0$, the density $\Psi_{\lambda_j}(x)$ is just a constant and we obtain the uniform measure $\frac{2}{(q+1)N}\,\dd x$ on $\mathbf{G}_N$ (see \S~\ref{sec:consequences} for more precision). In general, $\Psi_{\lambda_j}(x)$ is essentially the density of the spectral measure of the infinite quantum tree. The equidistribution in case $U=\alpha=0$ holds precisely because this spectral measure is constant in this case, as $x$ varies along the tree.

Our paper is organized as follows. In Section~\ref{sec:mainres} we introduce our quantum graphs and discuss the main results. In Section~\ref{sec:relations}, we construct discrete functions on the directed bonds of $G_N$, which turn out to be eigenfunctions of a non-backtracking operator $\cB$. We then prove our main result in Section~\ref{sec:equiregproof}. As in the case of combinatorial graphs \cite{A,ALM,AS2}, one can also ask about the behaviour of the eigenfunction correlators $\overline{\psi_j^{(N)}(x)}\psi_j^{(N)}(y)\,\dd x\,\dd y$. For this purpose, we provide a more general quantum ergodicity result, Theorem~\ref{thm:integral}, involving integral operators.  
This result is proved in Section~\ref{sec:integralop}.

We plan to address the general criterion for quantum ergodicity in such a bounded interval regime in a future work. It should be expected that the proof will become much more difficult and technical in this case. A very interesting application would be to prove quantum ergodicity for weak random perturbations of the equilateral model we study here, whether in the edge lengths or the coupling constants. In fact, it was shown in \cite{ASW06} that the AC spectrum of the Laplace operator on the equilateral quantum tree $\mathbf{T}_q$, is stable under weak random perturbations of the edge lengths.

\section{Presentation of the results}\label{sec:mainres}

\subsection{Background on quantum graphs}\label{sec:DefQG}

Let $G=(V,E)$ be a (combinatorial) graph, and fix $L>0$. The associated equilateral metric graph, denoted by $\mathbf{G}$, is obtained by identifying each edge with $[0,L]$. In doing so, we fix an orientation on each $e\in E$, identifying $0$ with an origin $o_e\in V$ and $L$ with a terminus $t_e\in V$. We shall write $d_\mathbf{G}$ for the distance on $\mathbf{G}$ induced by the usual distance on $[0,L]$.

A \emph{quantum graph} is a self-adjoint differential operator acting on the Hilbert space $\mathcal{H}:=\mathop\bigoplus_{e\in E} L^2[0,L]$. We consider the following Schr\"odinger operator.

Let $U\in L^\infty(0,L)$ be a real-valued potential satisfying
\[
U(L-x) = U(x) \,.
\]
We define an operator $H_\mathbf{G}$ acting on $\psi= (\psi_e)_{e\in E}\in  \mathop\bigoplus_{e\in E} W^{2,2}(0,L)\subset \mathcal{H}$ by 
\begin{equation}\label{e:H}
(H_{\mathbf{G}}\psi_e)(x_e) = -\psi''_e(x_e) + U(x_e)\psi_e(x_e).
\end{equation}

Thus defined, $H_{\mathbf{G}}$ is not essentially self-adjoint: we need to impose suitable boundary conditions at each vertex. In this paper, we will consider the generalized \emph{Kirchhoff boundary conditions} with parameter $\alpha\in \R$ given by
\begin{itemize}
\item \textbf{Continuity}: For all $v\in V$ and all $e,e'\in E$, we have $\psi_e(t_e) = \psi_{e'}(o_{e'}) =:\psi(v)$ if $t_e=o_{e'}=v$.
\item \textbf{Current conservation:} For all $v\in V$, 
\begin{equation*}
\sum_{e:o_e=v}\psi_e'(0) - \sum_{e:t_e=v} \psi_e'(L) = \alpha \psi(v).
\end{equation*}
\end{itemize}

If $G$ is finite, $H_{\mathbf{G}}$ is self-adjoint on the domain of functions in $\bigoplus_{e\in E}W^{2,2}(0,L)$ satisfying the Kirchhoff conditions. When $G=\T_q$ is the infinite tree, $H_{\mathbf{T}_q}$ is essentially self-adjoint on the set of compactly supported functions in $\bigoplus_{e\in E}W^{2,2}(0,L)$ satisfying the Kirchhoff conditions \cite{Car97}, and can hence be uniquely extended to a self-adjoint operator. We refer the reader to \cite{BeKu13} for more details on the construction and properties of quantum graphs.

In the sequel, we will consider sequences of finite graphs $G_N=(V_N,E_N)$, which we will suppose to be $(q+1)$-regular for some $q\geq 2$, meaning that
all vertices are connected to exactly $q+1$ neighbors. 
However, to lighten a bit notations and proofs, the parameters $L$ and $\alpha$, as well as the potential $U$, will not depend on $N$.

\subsubsection*{Eigenfunctions on the edges}
Fix an edge $e\in E_N$ and consider the eigenproblem
\begin{equation}\label{e:eigenproblem}
-\psi_e'' + U\psi_e = \lambda \psi_e
\end{equation}
for $\lambda\in \C$ and $\psi_e\in W^{2,2}(0,L)$. On this edge $e\equiv [0,L]$, choose a basis of solutions $C_{\lambda}(x)$, $S_{\lambda}(x)$ of \eqref{e:eigenproblem} satisfying
\[
\begin{pmatrix} C_{\lambda}(0) & S_{\lambda}(0) \\ C'_{\lambda}(0) & S'_{\lambda}(0) \end{pmatrix} = \begin{pmatrix} 1 & 0 \\ 0 & 1 \end{pmatrix} .
\]
Note that if $U\equiv 0$, we may take $C_{\lambda}(x) = \cos\sqrt{\lambda} x$ and $S_{\lambda}(x) = \frac{\sin\sqrt{\lambda}x}{\sqrt{\lambda}}$.

Denote $c(\lambda) = C_{\lambda}(L)$, $s(\lambda) = S_{\lambda}(L)$, $c'(\lambda) = C'_{\lambda}(L)$ and $s'(\lambda)=S'_{\lambda}(L)$.
Any solution $\psi_e$ of \eqref{e:eigenproblem} satisfying $\psi_e(0,\lambda) = c_1$ and $\psi'_e(0,\lambda)=c_2$ will have at $L$ the values
\begin{equation}\label{e:eiganmat}
\begin{pmatrix} \psi_e(L,\lambda)\\\psi'_e(L,\lambda)\end{pmatrix} = M_{\lambda}\begin{pmatrix} c_1\\c_2\end{pmatrix} \qquad \text{where } M_{\lambda} = \begin{pmatrix} c(\lambda) & s(\lambda) \\ c'(\lambda) & s'(\lambda) \end{pmatrix} .
\end{equation}

By the Wronskian identity, we have 
\begin{equation}\label{eq:Wronskian}
c(\lambda)s'(\lambda)- s(\lambda) c'(\lambda)=1.
\end{equation}

Since $U$ is symmetric, we have the following identity (cf. \cite[Lemma 3.1]{Car97}):
\begin{equation}\label{eq:FromCtoS}
c(\lambda)=s'(\lambda)
\end{equation}
Indeed, $x\mapsto C_{\lambda}(L-x)$ and $x\mapsto s'(\lambda)C_{\lambda}(x) - c'(\lambda)S_{\lambda}(x)$ are both solutions of (\ref{e:eigenproblem}), with the same value and derivative at $x=L$, so that they must be equal. Evaluating these functions at $x=0$ gives (\ref{eq:FromCtoS}). Similarly, we have
\begin{equation}\label{e:simplif}
s(\lambda)C_{\lambda}(x)-c(\lambda)S_{\lambda}(x) = S_{\lambda}(L-x).
\end{equation}

\subsection{Spectral theory on the infinite tree}
Let $H_{\mathbf{T}_q}$ be the operator \eqref{e:H} on the equilateral $(q+1)$-regular quantum tree $\mathbf{G}=\mathbf{T}_q$, with the Kirchhoff boundary conditions.

It is known \cite{Car97} that
\begin{equation}\label{e:spech}
\sigma(H_{\mathbf{T}_q}) = \sigma_1 \cup \sigma_2 \,,
\end{equation}
where
\begin{equation}\label{e:spechh}
\sigma_1 = \left\{\lambda\in \R : \left| (q+1)c(\lambda) + \alpha s(\lambda)\right| \le 2\sqrt{q} \right\} ,
\end{equation}
\[
\sigma_2 = \left\{\lambda\in \R : s(\lambda) = 0 \right\} .
\]
Moreover, assuming $q\ge 2$,
\[
\sigma_{\mathrm{ac}}(H_{\mathbf{T}_q}) = \sigma_1\,, \quad \sigma_{\mathrm{pp}}(H_{\mathbf{T}_q}) = \sigma_2 \quad \text{and} \quad \sigma_1\cap \sigma_2 = \emptyset \,.
\]
In general, $\sigma_2$ is a discrete set, $\sigma_2 \subset \sigma(D)$, where $D$ is the operator $D\psi = -\psi''+U\psi$ with Dirichlet boundary conditions on $[0,L]$ and points of $\sigma_2$ are eigenvalues of $H_{\mathbf{T}_q}$ with infinite multiplicity. Moreover, $\sigma(H_{\mathbf{T}_q})$ has a band structure with infinitely many gaps \cite{Pan06}. In the special case $\alpha=0$ and $U\equiv0$, we denote $-\Delta_{\mathbf{T}_q}=H_{\mathbf{T}_q}$, and have
\[
\sigma_{\mathrm{ac}}(-\Delta_{\mathbf{T}_q}) = \bigcup_{n=1}^{\infty} \left[ \left(\frac{(n-1)\pi + \theta}{L}\right)^2, \left(\frac{n\pi -\theta}{L}\right)^2\right] \quad \text{and} \quad \sigma_{\mathrm{pp}}(-\Delta_{\mathbf{T}_q}) = \left\{\left(\frac{n\pi}{L}\right)^2\right\}_{n=1}^{\infty} \,,
\]
where $\theta = \arccos \frac{2\sqrt{q}}{q+1}\in(0,\frac{\pi}2)$.

We say that $I\subset \sigma_{\mathrm{ac}}(H_{\mathbf{T}_q})$ is \emph{in a fixed band} if the map $I\ni \lambda \mapsto (q+1)c(\lambda) + \alpha s(\lambda)$ is injective.

\subsection{Presentation of our results}\label{sec:Presentation}
Consider a sequence of $(q+1)$-regular combinatorial graphs $G_N= (V_N,E_N)$, with $|V_N|=N$, satisfying the following assumptions:

\medskip

\textbf{(EXP)} The sequence $(G_N)$ is a family of expanders. That is, there exists $\beta>0$ such that the spectrum $\sigma(\frac{\cA_{G_N}}{q+1})\subset [-1+\beta,1-\beta]\cup\{1\}$ for all $N$, where $\cA_{G_N}$ is the adjacency matrix
of the graph $G_N$.

\medskip

\textbf{(BST)} For all $r>0$,
\[
\lim_{N\to \infty} \frac{\#\{x\in V_N : \rho_{G_N}(x)<r\}}{N} = 0
\]
where $\rho_{G_N}(x)$ is the largest $\rho$ such that the sub-graph contained
in a ball of radius $\rho$ centered at $x$ has no closed cycles.

\medskip

These assumptions were also needed to prove quantum ergodicity for discrete graphs in \cite{A,AS2}. They are known to be ``generic'', in the sense that a regular graph picked at random will typically satisfy these assumptions. There are also explicit examples of $(G_N)$ satisfying both assumptions \cite{LPS88,BoGa08}. Assumption \textbf{(BST)}  means that the graphs $(G_N)$ have few short cycles, equivalently, they converge to the combinatorial $(q+1)$-regular tree $\T_q$ in the sense of Benjamini-Schramm \cite{BS}.

We may now present our main result,  which essentially says that in the regions where $H_{\mathbf{T}_q}$ is spectrally delocalized, it also has a form of spatial delocalization. Namely, if we consider a sequence of quantum graphs converging to $\mathbf{T}_q$, then the eigenfunctions in any bounded interval within the AC spectrum of $H_{\mathbf{T}_q}$ are quantum ergodic when the graph is large enough. To clarify this result, we present it in the form of three theorems of growing generality. Theorem~\ref{cor:equireg} discusses the special case where the observables are constant on the edges. General bounded observables are considered in Theorem~\ref{thm:equireg}. Finally, Theorem~\ref{thm:integral} studies the case where the observables are replaced by integral kernels. 

\begin{thm}\label{cor:equireg}
Assume $(G_N)$ satisfies \emph{\textbf{(EXP)}} and \emph{\textbf{(BST)}}.

Let $(\psi_n^{(N)})_{n\in\N}$ be an orthonormal basis of eigenfunctions of $H_{\mathbf{G}_N}$ and let $(\lambda^{(N)}_n)$ be the corresponding sequence of eigenvalues.  Let $I$ be an open interval such that $\overline{I}\subset \sigma_{\mathrm{ac}}(H_{\mathbf{T}_q})$ lies in a fixed band of AC spectrum, and denote by $N(I)$ the count of $\lambda^{(N)}_n\in I$.

Let $a : E_N\rightarrow [-1,1]$. Then
\begin{equation}\label{eq:QE}
\lim_{N\to\infty} \frac{1}{N(I)} \sum_{\lambda_n^{(N)}\in I} \left|\langle \psi_n^{(N)}, a\psi_n^{(N)}\rangle -\langle a\rangle \right|^2 =0 \,,
\end{equation}
where $\langle \psi^{(N)}_n, a\psi^{(N)}_n\rangle = \sum_{e\in E_N} a(e) \int_0^L |\psi^{(N)}_n(x_e)|^2\,\dd x_e$ and
\begin{displaymath}
\langle a\rangle = \frac{1}{|E_N|}\sum_{e\in E_N} a(e).
\end{displaymath}
\end{thm}

We postpone discussing the consequences of this theorem until \S~\ref{sec:consequences}. In order to describe the limiting distribution of the $|\psi_n^{(N)}(t)|^2\,\dd t$ inside the edges, we now replace the locally constant observables $a(x_e) = a(e)$ by general observables $f(x_e)$ as follows.

\begin{thm}\label{thm:equireg}
Assume $(G_N)$ satisfy \emph{\textbf{(EXP)}} and \emph{\textbf{(BST)}}.

Let $(\psi_n^{(N)})_{n\in\N}$ be an orthonormal basis of eigenfunctions of $H_{\mathbf{G}_N}$ and let $I$ be an open interval such that $\overline{I}\subset \sigma_{\mathrm{ac}}(H_{\mathbf{T}_q})$ lies in a fixed band of AC spectrum. Then for any function $f=(f_e)\in \mathop\bigoplus_{e\in E_N}L^2[0,L]$ satisfying $|f(x)| \le 1$, we have
\begin{equation}\label{e:main'}
\lim_{N\to\infty} \frac{1}{N(I)} \sum_{\lambda_n^{(N)}\in I} \left|\langle \psi_n^{(N)}, f\psi_n^{(N)}\rangle -\langle f\rangle_{\lambda_n^{(N)}} \right|^2 =0 \,,
\end{equation}
where $\langle \psi_n, f\psi_n\rangle = \sum_{e\in E_N} \int_0^L f_e(x_e) |\psi_n(x_e)|^2\,\dd x_e$,
\begin{equation}\label{e:weightedav}
\langle f \rangle_{\lambda_n} = \frac{1}{N} \sum_{e\in E_N} \int_0^L f_e(x_e) \Psi_{\lambda_n}(x_e) \,\dd x_e \,,
\end{equation}
\[
\Psi_{\lambda}(x_e) = \frac{1}{\kappa_{\lambda}} \cdot \frac{\Im G_{\mathbf{T}_q}^{\lambda+\ii 0}(\tilde{x}_e,\tilde{x}_e)}{\Im G^{\lambda+\ii 0}_{\mathbf{T}_q}(o,o)}\,,
\]
and if $w(\lambda) := (q+1)c(\lambda) + \alpha s(\lambda)$, we have
\begin{equation}\label{e:kappalambda}
\kappa_{\lambda}= \frac{q+1}{s^2(\lambda)}\int_0^L S^2_{\lambda}(t)\,\dd t + \frac{w(\lambda)}{s^2(\lambda)}\int_0^LS_{\lambda}(L-t)S_{\lambda}(t)\,\dd t\, .
\end{equation}
Here, $\tilde{x}_e$ is a lift of $x_e$ to $\mathbf{T}_q$ (the universal cover of $\mathbf{G}_N$).
\end{thm}

In the previous theorem, $G_{\mathbf{T}_q}^{\gamma}(x,y) = (H_{\mathbf{T}_q}-\gamma)^{-1}(x,y)$ is the Green's function of $H_{\mathbf{T}_q}$, $o$ is an arbitrary vertex, and we denoted $G^{\lambda+\ii 0}(x,y) = \lim_{\eta\downarrow 0} G^{\lambda+\ii\eta}(x,y)$. Note that this Green's function is completely explicit, see \eqref{eq:greensfunction} and Appendix~\ref{app:a}.

Theorem \ref{cor:equireg} follows directly from Theorem \ref{thm:equireg} and identity \eqref{eq:3}, which says that $\int_0^L\Psi_{\lambda_n}(t)\,\dd t = \frac{2}{q+1}$, and thus $\frac{\int_0^L\Psi_{\lambda_n}(t)\,\dd t}{N} = \frac{1}{|E_N|}$.

Note that the observable $f$ depends on $N$, though we do not indicate this explicitly in the notation. In fact, one chooses an observable $f=f_N$ for each graph $\mathbf{G}_N$.

\medskip

\subsubsection*{Quantum ergodicity for integral operators}

We now present a more general quantum ergodicity result whose aim is to study the eigenfunction correlator $\overline{\psi_n^{(N)}(x)}\psi_n^{(N)}(y)\,\dd x\,\dd y$. This also allows us to gain a better understanding of the limiting density $\Psi_{\lambda}$.

In general, an integral operator on $\mathop\bigoplus_e L^2[0,L_e]$ takes the form
\[
(K\psi)(x_e) = \sum_{e'\in E_N} \int_0^{L_{e'}} K_{e,e'}(x_e,y_{e'})\psi_{e'}(y_{e'})\,\dd y_{e'} \,.
\]
In the following we consider integral operators whose kernels vanish if $d(e,e')$ is large. For convenience, we consider directed edges and consider operators of the form
\begin{equation}
(K_k\psi)(x_b) = \sum_{(b_1,\dots,b_k)_b} \int_0^{L} K_{b_1,b_k}(x_{b_1},y_{b_k})\psi_{b_k}(y_{b_k})\,\dd y_{b_k}\,,
\end{equation}
where $b_1 := b$ and the sum runs over non-backtracking paths of length exactly $k$ whose first vertices are $(o_b,t_b)$. Given such $K_k$, we have
\begin{equation}\label{e:uneq}
2\left\langle \psi_n^{(N)}, K_k \psi_n^{(N)}\right\rangle = \sum_{(b_1;b_k)\in B_k} \int_0^L\int_0^L K_{b_1,b_k}(x_{b_1},y_{b_k})\overline{\psi_n^{(N)}(x_{b_1})}\psi_n^{(N)}(y_{b_k})\,\dd x_{b_1} \dd y_{b_k}\,,
\end{equation}
where the sum now runs over all $k$-paths $(b_1;b_k)\equiv(x_0;x_k)$ in $G_N$.
The factor $2$ in the left-hand side of \eqref{e:uneq} arises because we
sum over each bond twice. We shall prove that

\begin{thm}\label{thm:integral}
Under the assumptions of Theorem~\ref{thm:equireg}, if the integral kernels satisfy $|K_{b_1,b_k}(x_{b_1},y_{b_k})|\le 1$, then
\[
\lim_{N\to \infty} \frac{1}{N(I)} \sum_{\lambda_n^{(N)}\in I} \left| \langle \psi_n^{(N)}, K_k \psi_n^{(N)}\rangle - \langle K_k\rangle_{\lambda_n^{(N)}}\right|^2 = 0\,,
\]
where
\[
\langle K_k\rangle_{\lambda_n^{(N)}} = \frac{1}{N}\sum_{(b_1;b_k)\in B_k} \int_0^L \int_0^L K_{b_1,b_k}(x_{b_1},y_{b_k})\Psi_{\lambda_n^{(N)},k}(x_{b_1},y_{b_k})\,\dd x_{b_1}\,\dd y_{b_k} \,,
\]
and for $x_{b_1}\in b_1$, $y_{b_k}\in b_k$, we have
\begin{equation}\label{e:psienfin}
\Psi_{\lambda,k}(x_{b_1},y_{b_k}) = \frac{1}{2\kappa_{\lambda}}\cdot 
\frac{\Im G^{\lambda+\ii 0}_{\mathbf{T}_q}(\tilde{x}_{b_1},\tilde{y}_{b_k})}{\Im G_{\mathbf{T}_q}^{\lambda
+\ii 0}(o,o)} \,.
\end{equation}
Here, $\tilde{x}_{b_1}$ and $\tilde{y}_{b_k}$ are lifts of $x_{b_1}$ and $y_{b_k}$ to $\mathbf{T}_q$ such that $d(\tilde{x}_{b_1},\tilde{y}_{b_k})= d(x_{b_1},y_{b_k})$.
\end{thm}

\subsection{Consequences}\label{sec:consequences}

Using Markov's inequality, it follows from Theorem~\ref{cor:equireg} that for any $\eps>0$, we have
\[
\lim_{N\to\infty} \frac{1}{N(I)}\#\left\{\lambda_n^{(N)}\in I : \left|\left\langle \psi_n^{(N)},a\psi_n^{(N)}\right\rangle - \langle a\rangle\right|>\eps\right\} =0 \,.
\]
In particular, we obtain that for fixed large $N$, we have 
\begin{equation}\label{eq:ApproxEquid}
\sum_{e\in E_N} a(e) \int_0^L |\psi^{(N)}_n(x_e)|^2\,\dd x_e\approx \frac{1}{|E_N|}\sum_{e\in E_N} a(e)
\end{equation}
 for most $\lambda_n^{(N)}\in I$. Formula \eqref{eq:ApproxEquid} is justified if $a$ is supported on a macroscopic part of $\mathbf{G}_N$, say $|\supp a| = c \cdot |\mathbf{G}_N|$ for some $0<c\le 1$. In this case, the error term, which is related to the presence of cycles, is sufficiently small for typical graphs. However, the control we have on the remainder in the limit (\ref{eq:QE}) is not good enough to justify (\ref{eq:ApproxEquid}) when $a$ has a small support (for instance, when $a$ is supported on a single edge). Similar difficulties arise for combinatorial graphs \cite{ALM,AS4}.

For general observables $f$, we similarly deduce from Theorem~\ref{thm:equireg} that
\begin{equation}\label{e:approxequi}
\int_{\mathbf{G}_N} |\psi_n^{(N)}(x)|^2f(x)\,\dd x \approx \frac{1}{N}\int_{\mathbf{G}_N} f(x)\Psi_{\lambda_n^{(N)}}(x)\,\dd x
\end{equation}
for most $\lambda_n^{(N)}\in I$ and observables with large support. One subtlety one should not forget is that when we say ``for most $\lambda_n^{(N)}$'', the set of indices for which this holds true depends on the observable $f$.

The limit measure turns out to be the uniform measure in the special case 
$U\equiv0$ and $\alpha=0$, as we report in Appendix~\ref{app:a2}. So in that special case we get for \emph{any} observable $f^{(N)}=f$ with large support that $\int_{\mathbf{G}_N} |\psi_n^{(N)}(x)|^2f(x)\,\dd x \approx \frac{1}{|\mathbf{G}_N|}\int_{\mathbf{G}_N} f(x)\,\dd x$ for most $\lambda_n^{(N)}\in I$. In other words, most eigenfunctions equidistribute in the large graph limit (in a weak sense). Such an equidistribution holds because in this special case, for any \emph{fixed} $\lambda$, the density of the spectral measure of the infinite tree $\frac{1}{\pi} \Im G^{\lambda+\ii 0}(x,x)$ is constant as $x$ moves along the tree, so that $\Psi_{\lambda_n^{(N)}}(x)$ is a constant. In general, \eqref{e:approxequi} says that $|\psi_n^{(N)}(x)|^2$ varies along the tree essentially like this density does (at the fixed $\lambda_n^{(N)}$). Note that by \eqref{eq:2}, $x\mapsto \Psi_{\lambda}(x)$ is continuously differentiable and symmetric on each edge~: $\Psi_{\lambda}(L-x)=\Psi_{\lambda}(x)$.

Curiously, the coefficient $\frac{1}{\kappa_{\lambda_n}}$ is the $\ell^2$-norm of the function $\psi_n$ restricted to the vertices (see \eqref{e:normpsidis}).

\medskip

Theorem~\ref{thm:integral} can similarly be interpreted as saying that when $N$ gets large enough, the correlation $\overline{\psi_n^{(N)}(x_{b_1})}\psi_n^{(N)}(y_{b_k})\,\dd x_{b_1}\,\dd y_{b_k}$ approaches $\frac{1}{N}\cdot \frac{1}{\kappa_{\lambda}} \cdot  \frac{\Im G^{\lambda_n}_{\mathbf{T}_q}(\tilde{x}_{b_1},\tilde{y}_{b_k})}{\Im G_{\mathbf{T}_q}^{\lambda_n}(o,o)}\,\dd x_{b_1}\,\dd y_{b_k}$.

Note that Theorem~\ref{thm:integral} formally generalizes Theorem~\ref{thm:equireg}. In particular, the $|\psi_n^{(N)}(t)|^2\,\dd t$ equidistributes if and only if the diagonal $\Psi_{\lambda_n,1}(t,t)=\mathrm{constant}$.

\begin{rem}\label{rem:Dirichlet}
We can also test for quantum ergodicity in several bands of AC spectra simultaneously. In fact, if $I$ is a finite union of open intervals $I_r$, each lying in a band, then the variance \eqref{e:main'} of each $I_r$ vanishes, so the same holds for the sum on $I$.

If we replace $\frac{1}{N(I)}$ by $\frac{1}{N}$ in \eqref{e:main'}, we can deduce moreover that the variance \eqref{e:main'} also vanishes for $I$ between the AC bands, as long as $I\cap \sigma_2=\emptyset$. In this case, no real proof is needed; this simply follows because the number of eigenvalues in such regions becomes negligible as the graph grows large. See Remark~\ref{rem:Weyl} for details.

Finally, one does not expect delocalization for $\lambda\in\sigma_2$. For such $\lambda$, one can use  the Dirichlet condition $s(\lambda)=0$ to construct eigenfunctions of $H_{\mathbf{G}_N}$ which are supported on cycles with an even number of edges, essentially by taking $S_{\lambda}(x)$ on each edge of the cycle. See \cite[Figure 4]{kuc:qgII} for an illustration. Such eigenfunctions vanish identically elsewhere on the graph, quite the opposite of being ``uniformly distributed''. See \cite[Section 5]{kuc:qgII} and \cite{CdV15} for further discussion on eigenfunctions of small support.
\end{rem}

\begin{rem}\label{rem:cycle}
One may wonder if we can hope for a stronger ``quantum unique ergodicity'' result, namely that $\int_{\mathbf{G}_N} |\psi_n^{(N)}(x)|^2f(x)\,\dd x \approx \frac{1}{|\mathbf{G}_N|}\int_{\mathbf{G}_N} f(x)\,\dd x$ for \emph{all} $\lambda_n^{(N)}\in I$ and all observables $f$.

To illustrate the subtlety of this question, let us discuss the very easy case where $q=1$ and $U=\alpha=0$. In other words, $G_N$ is just an $N$-cycle. Under Kirchhoff conditions, the Laplacian on $\mathbf{G}_N$ is the same as the Laplacian on the interval $\mathbf{I}_N=[0,NL]$ with periodic boundary conditions. Note that this case does not enter the framework of our theorem, since these graphs are not expanders.

An orthonormal basis of eigenfunctions of $-\Delta_{\mathbf{I}_N}$ is given by $e_k(x) = \frac{\exp(\frac{2\pi \ii k x}{NL})}{\sqrt{NL}}$, $k\in \Z$. For this basis, quantum ergodicity, and even quantum unique ergodicity, are trivial; in fact $\int |e_k(x)|^2 f(x) \dd x = \frac{1}{NL}\int f(x)\,\dd x$.

A more interesting basis is given by $(v_j)_{j\in \N}$, where 
\begin{displaymath}
  v_0(x)=\frac{1}{\sqrt{NL}},\quad v_{2j}(x) = \sqrt{\frac{2}{NL}}\sin\left(\frac{2j\pi x}{NL}\right),\quad \text{and}\quad v_{2j+1}(x) = \sqrt{\frac{2}{NL}}\cos\left(\frac{2j\pi x}{NL}\right). 
\end{displaymath}
In this case we have, $\int |v_{2j}(x)|^2 f(x)\,\dd x = \frac{1}{NL}\int f(x)\,\dd x - \frac{1}{NL}\int f(x)\cos(\frac{4j\pi x}{NL})\,\dd x$ as well as $\int |v_{2j+1}(x)|^2f(x)\,\dd x = \frac{1}{NL} \int f(x)\,\dd x + \frac{1}{NL}\int f(x)\cos(\frac{4j\pi x}{NL})\,\dd x$. Hence,
\[
\frac{1}{N(I)}\sum_{\lambda_j\in I} \left|\langle v_j, fv_j\rangle - \langle f\rangle\right|^2 = \frac{1}{N(I)}\sum_{\lambda_j\in I} \left|\frac{1}{NL}\int f(x) \cos\Big(\frac{4j\pi x}{NL}\Big)\,\dd x\right|^2 \,.
\]
But $\sum_{j\in \N} |\sqrt{\frac{2}{NL}}\int \cos(\frac{4j\pi x}{NL})f(x)\,\dd x|^2 = \sum_{j\in \N} |\langle v_{4j+1}, f\rangle|^2 \le \|f\|_{L^2(\mathbf{I}_N)}^2 \le NL$, assuming $|f(x)|\le 1$. Hence,
\[
\frac{1}{N(I)}\sum_{\lambda_j\in I} \left|\langle v_j, fv_j\rangle - \langle f\rangle\right|^2 \le \frac{1}{N(I)} \cdot \frac{1}{NL} \cdot NL = \frac{1}{N(I)} \to 0
\]
as $N\to \infty$, since $N(I) \to \infty$. In fact, $\lambda_j\in [a,b]$ iff $(\frac{2\pi j}{NL})^2\in [a,b]$, i.e. $j\in [\frac{NL}{2\pi}\sqrt{a}, \frac{NL}{2\pi}\sqrt{b}]$, so $N(I) = c_I \cdot N$. This proves (\ref{e:main'}) for this basis.

However, note that for $\lambda_j\in I$, so that $j = \frac{NL}{2\pi}c_j$ for some $c_j\in [\sqrt{a},\sqrt{b}]$, we have $\int |v_{2j}(x)|^2f(x) = \frac{2}{NL}\int f(x)\sin^2(c_jx)\,\dd x$. We may choose observables such that this is not close to the uniform average. In fact, taking $f(x)=\cos(2c_j x)$ yields the error $\frac{1}{NL}\int f(x)\cos(2c_j x) = \frac{1}{2}$.
Therefore, in this case, we do not have $\int_{\mathbf{G}_N} |\psi_n^{(N)}(x)|^2f(x)\,\dd x \approx \frac{1}{|\mathbf{G}_N|}\int_{\mathbf{G}_N} f(x)\,\dd x$ for \emph{all} $\lambda_n^{(N)}\in I$ and all observables $f$.

This example shows that even in very simple graphs, a naive statement of QUE will not hold in the large $N$ limit. The fact that the eigenfunctions and observables vary with $N$ makes things quite complicated, so the pertinent formulation is not entirely clear at this point. Note that in contrast, for fixed $N$, say $N=1$, the same basis satisfies QUE in the high energy limit. In that case, $|\langle v_j, fv_j\rangle - \langle f \rangle| = |\frac{1}{L} \int_0^L f(x) \cos(4j\pi x/L)\, \mathrm{d}x| \to 0$ as $j\to \infty$ by the Riemann-Lebesgue lemma.

\end{rem}

\section{Relationship between the continuous and discrete problems}\label{sec:relations}

In this section, we derive some preliminary relations between the eigenfunctions of $H_{\mathbf{G}}$ and $\cA_G$ in \S~\ref{sec:nonback}, and the Green's functions of $H_{\mathbf{T}_q}$ and $\cA_{\T_q}$ in \S~\ref{sec:greenkernel}, which will be used later to prove the main result. We conclude this section by recalling the basic properties of the spherical function of $\T_q$ in \S~\ref{sec:spherico}.

\subsection{Non-backtracking eigenfunctions}\label{sec:nonback}

Given $\gamma\in\C$, we set
\begin{equation}\label{e:wlambda}
w(\gamma) := (q+1)c(\gamma)+\alpha s(\gamma) \,.
\end{equation}
As in \cite{Car97}, we consider
\begin{equation}\label{e:mupm}
\mu^{\pm}(\gamma) := \frac{w(\gamma)\pm \sqrt{w(\gamma)^2-4q}}{2q} .
\end{equation}
This function arises in the study of transfer matrices on the quantum tree. Note that for $\lambda\in \sigma_{\mathrm{ac}}(H_{\mathbf{T}_q})$, the $\mu^{\pm}(\lambda)$ are complex conjugate (see \eqref{e:spechh}), with $|\mu^{\pm}(\lambda)|^2=\frac{1}{q}$. In general,
\begin{equation}\label{e:muplusmoins}
\mu^+(\gamma)\mu^-(\gamma) = \frac{1}{q} \,, \qquad q\left(\mu^+(\gamma) + \mu^-(\gamma)\right) = w(\gamma) \,.
\end{equation}

If $(\lambda_j^{(N)})_{j\in\N}\subset \R$ is the discrete spectrum of $H_{\mathbf{G}_N}$, we shall denote $\mu_j^{\pm} := \mu^{\pm}(\lambda_j^{(N)})$.  We drop the superscripts $(N)$ from notations where it will not
lead to confusion.

Fix $G=G_N$. Given an orthonormal basis $(\psi_j)$ of eigenfunctions of $H_{\mathbf{G}}$ with eigenvalues $\lambda_j$, we define the discrete function $\mathring{\psi}_j$ on $V_N$ by
\begin{equation}\label{eq:defNewPsi}
\mathring{\psi}_j(v) = \psi_j(v) \,.
\end{equation}
This is well-defined as $\psi_j$ satisfies Kirchhoff conditions.

Let $B$ be the set of directed edges of $G_N$, so $|B| = 2\,|E_N| = N(q+1)$. Each edge $b\in B$ has an origin $o_b\in V_N$ and a terminus $t_b\in V_N$.
We denote by $\hat{b}$ the reversal of $b$.

Consider the \emph{non-backtracking operator} $\cB$ defined on $\ell^2(B)$ by
\[
(\cB f)(b) = \sum_{b^+ \in \cN_b^+} f(b^+) \,,
\]
where $\cN_b^+$ is the set of edges $b'$ with $o_{b'}=t_b$ and $b'\neq \hat{b}$, i.e.\ the set of outgoing edges from $b$. Define $\tau_{\pm}:\C^{V_N} \to \C^B$ by
\[
(\tau_+\psi)(b) = \psi(t_b) \qquad \text{and} \qquad (\tau_-\psi)(b) = \psi(o_b) \,.
\]
Inspired by \cite{A}, we define 
\begin{equation}\label{e:fjfjtoile}
f_j = \tau_+ \mathring{\psi}_j - \mu^-_j \tau_- \mathring{\psi}_j \qquad \text{and} \qquad f_j^{\ast} = \tau_-\mathring{\psi}_j - \mu^-_j \tau_+ \mathring{\psi}_j \,.
\end{equation}

\begin{lem}\label{lem:nonback}
The following properties hold:
\begin{equation}\label{e:discreteeigen}
\cA_{G} \mathring{\psi}_j = w(\lambda_j) \mathring{\psi}_j\,,
\end{equation}
\begin{equation}\label{e:nonback}
\mu^-_j \cB f_j = f_j\qquad \text{and}\qquad \mu^-_j \cB^{\ast} f_j^{\ast} = f_j^{\ast} \,.
\end{equation}
\end{lem}
\begin{proof}
Let $v\in V_N$ and $b\in B$ be such that $v=t_b$. We write 
\[
(\cA_{G}\mathring{\psi}_j)(t_b) = \psi_j(o_b) + \sum_{b^+\in \cN_b^+} \psi_j(t_{b^+}).
\]
Recall that $\psi_j(t_{b^+}) = \psi_j(o_{b^+})c(\lambda_j) + \psi_j'(o_{b^+})s(\lambda_j)$. By the Kirchhoff conditions, $\psi_j(o_{b^+}) = \psi_j(t_b)$ and $\sum_{b^+\in \cN_b^+} \psi_j'(o_{b^+}) = \psi_j'(t_b) + \alpha \psi_j(t_b)$. We thus get 
\[
(\cA_{G}\mathring{\psi}_j)(t_b) = \psi_j(o_b) + [qc(\lambda_j) + \alpha s(\lambda_j)]\psi_j(t_b) + s(\lambda_j) \psi_j'(t_b).
\]
 By \eqref{e:eiganmat}, $\begin{pmatrix} \psi_j(o_b)\\\psi_j'(o_b)\end{pmatrix} =M_{\lambda_j}^{-1}\begin{pmatrix}\psi_j(t_b)\\\psi_j'(t_b)\end{pmatrix}$. It follows that $\psi_j(o_b) = c(\lambda_j)\psi_j(t_b)-s(\lambda_j)\psi_j'(t_b)$. Recalling \eqref{e:wlambda}, this proves \eqref{e:discreteeigen}. Next, 
\begin{align*}
(\cB f_j)(b) &= \sum_{b^+\in \cN_b^+} [\psi_j(t_{b^+}) - \mu_j^-\psi_j(o_{b^+})] \\
&= (\cA \mathring{\psi}_j)(t_b) - \psi_j(o_b) - q\mu_j^- \psi_j(t_b) \\
&= [w(\lambda_j) - q\mu_j^-]\psi_j(t_b) - \psi_j(o_b).
\end{align*}
Using \eqref{e:muplusmoins}, we get $(\cB f_j)(b) = q\mu_j^+ \psi_j(t_b) - \psi_j(o_b) = \frac{1}{\mu_j^-}f_j(b)$.

Finally, to prove that $\cB^{\ast} f_j^{\ast} = \frac{1}{\mu_j^-}f_j^{\ast}$, note that if $\iota$ is the edge reversal, i.e. $(\iota f)(x_0,x_1)=f(x_1,x_0)$, then $\cB^{\ast} = \iota \cB \iota$ and $f_j^{\ast} = \iota f_j$, so we deduce that $\cB^{\ast}f_j^{\ast} =\frac{1}{\mu_j^-}f_j^{\ast}$.
\end{proof}

\subsection{Relationship between the continuous and discrete resolvents}\label{sec:greenkernel}

Let $G_{\mathbf{T}_q}^{\gamma}(x,y)$ be the integral kernel of the resolvent of $H_{\mathbf{T}_q}$ at energy $\gamma\in \C$, and let $G_{\cA_{\T_q}}^{\gamma}(v,w)$ be the Green's function $(\cA_{\T_q}-\gamma)^{-1}(v,w)$. This is well-defined, at least when $\gamma\in \C^+ = \{\Im z>0\}$, so that $\gamma$ is away from the spectrum. On the real line, we denote
\begin{equation*}
G^\lambda(x,y) := G^{\lambda+\mathrm{i}0}(x,y) :=
\lim_{\eta\downarrow0} G^{\lambda+\mathrm{i}\eta}(x,y),\qquad\lambda\in\R
\end{equation*}
when this limit exists. We use this notation for Green's functions of different objects, indicated with subscripts.

Explicit expressions for the Green's function on the quantum tree were
worked-out by Carlson in \cite{Car97}. One fixes $b_0\in \mathbf{T}_q$, $b_0\equiv [0,L]$, and define for $x\in b_0$,
\[
U_{\gamma}(x) = -s(\gamma)C_{\gamma}(x) + \left[c(\gamma)-q\mu^+(\gamma)\right] S_{\gamma}(x)\,,
\]
\[
V_{\gamma}(x) = s(\gamma)C_{\gamma}(L-x) - \left[c(\gamma)-q\mu^+(\gamma)\right]S_{\gamma}(L-x) \,.
\]
Given $x_0\in b_0$, using the notations of \cite{ASW06}, we let $\mathbf{T}_{x_0}^+$ and $\mathbf{T}_{x_0}^-$ be the subtrees produced by cutting $\mathbf{T}$ at $x_0$. Any $b^+\in \mathbf{T}_{x_0}^+$ has the form $\cB^nb_0$, which is an abusive notation to say that $b^+$ can be reached by a non-backtracking path of length $n$ from $b_0$. Similarly, any $b^-\in \mathbf{T}_{x_0}^-$ takes the form $b^-=\cB^{\ast m} b_0$ for some $m$. We then define the functions
\[
\begin{cases} U_{\gamma;x_0}^-(y) = \left(\mu^-(\gamma)\right)^m U_{\gamma}(\check{y}) &\text{if } y\in \cB^{\ast m} b_0,\\ V_{\gamma;x_0}^+(y) = \left(\mu^-(\gamma)\right)^n V_{\gamma}(\check{y}) &\text{if } y\in \cB^n b_0,\end{cases}
\]
where $\check{y}$ is the point in $b_0$ at the corresponding position to $y$.
Then $U_{\gamma;x_0}^-$ and $V_{\gamma;x_0}^+$ are solutions to the Kirchhoff problem on $\mathbf{T}_{x_0}^-$ and $\mathbf{T}_{x_0}^+$, respectively. Given $x,y\in \mathbf{T}$, choose $o,v$ such that $x,y\in \mathbf{T}_o^+ \cap \mathbf{T}_v^-$. Define
\begin{equation}\label{eq:greensfunction}
G^{\gamma}_{\mathbf{T}_q}(x,y) = \begin{cases} \frac{U_{\gamma;v}^-(x) V_{\gamma;o}^+(y)}{W^{\gamma}_{v,o}(x)} &\text{if } y\in \mathbf{T}_x^+,\\ \frac{U_{\gamma;v}^-(y) V_{\gamma;o}^+(x)}{W^{\gamma}_{v,o}(x)} &\text{if } y\in \mathbf{T}_x^-,\end{cases}
\end{equation}
where $W^{\gamma}_{v,o}(x)$ is the Wronskian
\[
W^{\gamma}_{v,o}(x) = V_{\gamma;o}^+(x)(U_{\gamma;v}^-)'(x) - (V_{\gamma;o}^+)'(x)U_{\gamma;v}^-(x) \,.
\]
Then it is shown in \cite{Car97} that $G^{\gamma}_{\mathbf{T}_q}(x,y)$ is precisely the Green kernel of $H_{\mathbf{T}_q}$.

In Appendix~\ref{app:a1}, we prove the following lemma~:

\begin{lem}\label{lem:Greenkernel}
Let $\gamma\in \C^+$ and $v,w$ be vertices in $\T_q$. Then we have
\[
G_{\mathbf{T}_q}^{\gamma}(v,w) = -s(\gamma) G_{\cA_{\T_q}}^{w(\gamma)}(v,w) \,.
\]
If $\lambda\in \sigma_{\mathrm{ac}}(H_{\mathbf{T}_q})$, we also have $G_{\mathbf{T}_q}^{\lambda+\ii 0}(v,w) = -s(\lambda) G_{\cA_{\T_q}}^{w(\lambda)+\ii 0}(v,w)$.
\end{lem}

\subsection{Spherical functions}\label{sec:spherico}
The spherical function of parameter $\lambda\in \R$ is given, for $d\in \N_0$, by
\begin{equation}\label{eq:DefSpherical}
\Phi_\lambda(d)= q^{-d/2} \left( \frac{2}{q+1} P_d\left(\frac{\lambda}{2\sqrt{q}}\right) + \frac{q-1}{q+1} Q_d\left(\frac{\lambda}{2\sqrt{q}}\right)\right),
\end{equation}
where $P_r(\cos\theta) = \cos r\theta$ and $Q_r(\cos\theta) = \frac{\sin(r+1)\theta}{\sin \theta}$ are the Chebyshev polynomials of the first and second kinds, respectively.

As is well-known, for $\lambda\in (-2\sqrt{q},2 \sqrt{q})$, we have
\begin{equation*}
\Phi_{\lambda}(d(v,w)) = \frac{\Im G_{\cA_{\T_q}}^{\lambda}(v,w)}{\Im G_{\cA_{\T_q}}^{\lambda}(o,o)},
\end{equation*}
 where $o\in \T_q$ is any vertex. This follows from the fact that $\Phi_{\lambda,v}(w) := \Phi_{\lambda}(d(v,w))$ is the unique (generalized) eigenfunction of $\cA_{\T_q}$ which is radial from $v$ and normalized as $\Phi_{\lambda,v}(v)=1$. So one just checks that the function on the right-hand side satisfies these properties. Here ``generalized'' means non-$\ell^2$ (the spectrum of $\cA_{\T_q}$ is purely AC).

Using Lemma~\ref{lem:Greenkernel}, we deduce that for $\lambda \in \sigma_{\mathrm{ac}}(H_{\mathbf{T}_q})$,
 \begin{equation}\label{eq:SphericalandResolvant}
 \Phi_{w(\lambda)}(k) = \frac{\Im G_{\mathbf{T}_q}^{\lambda}(v,w)}{\Im G_{\mathbf{T}_q}^{\lambda}(o,o)}
 \end{equation}
for vertices $v$, $w$ separated by a distance $k$ on the tree.
 
The spherical function is also the unique solution of the recursive relation
\begin{equation}\label{eq:RecursSpherical}
\Phi_\lambda(d) = \frac{1}{q}\left(\lambda\Phi_\lambda(d-1)- \Phi_\lambda(d-2)\right)
\end{equation}
with initial conditions
\begin{equation}\label{e:inisphe}
 \Phi_\lambda(1) = \frac{\lambda}{q+1} \qquad\text{and}\qquad  \Phi_\lambda(0) = 1.
\end{equation}

\section{Proof of the main theorem}\label{sec:equiregproof}

To prove Theorem \ref{thm:equireg} (which contains Theorem~\ref{cor:equireg}
as a special case) we need to show that the quantum variance
\begin{equation}
  \label{eq:quant_variance}
 \frac{1}{N(I)} \sum_{\lambda_n^{(N)}\in I} \left|\langle \psi_n^{(N)}, f\psi_n^{(N)}\rangle -\langle f\rangle_{\lambda_n^{(N)}} \right|^2   
\end{equation}
vanishes as $N\to\infty$. Recall that $\langle \psi_n,f\psi_n\rangle = \sum_{e\in E_N} \int_0^Lf_e(x_e)|\psi_n(x_e)|^2\,\dd x_e$ and $\langle f\rangle_{\lambda_n}$ is defined in \eqref{e:weightedav}.

Our strategy is as follows. We first show in \S~\ref{sec:discrired} that this quantum variance involving functions on the metric graph, can be bounded from above by certain variances defined on the \emph{discrete} graph. Let us introduce some notations to clarify this point. 

If $k\in \N$, we denote by $\mathscr{H}_k = \C^{B_k}$ the set of complex-valued functions on the set $B_k$ of (discrete) non-backtracking paths of length $k$ in $G=G_N$. By convention, $B_0=V_N$ and $B_1 = B$. 

Given $K^{\lambda}\in\mathscr{H}_k$, a family of observables depending on
$\lambda$, we define the \emph{discrete} quantum variance
\begin{equation}\label{e:vari}
  \vari(K^{\lambda}) = \frac{1}{N(I)}\sum_{\lambda_n\in I} \left|\langle
    \mathring{\psi}_n, K_G^{\lambda_n} \mathring{\psi}_n\rangle\right|^2 ,
\end{equation}
where the functions  $\mathring{\psi}_n:V_N\to\C$ were defined in \eqref{eq:defNewPsi} and
\begin{equation}\label{e:kg}
  \langle \varphi, K_G \psi \rangle = \sum_{(x_0;x_k)\in B_k}
  \overline{\varphi(x_0)}K(x_0;x_k)\psi(x_k) \,.
\end{equation}

Then the main result of \S~\ref{sec:discrired} appears as \eqref{eq:simple_bound}. The main difficulty with the resulting variances is that the discrete observables $K_{f,n}^{\lambda},J_{f,n}^{\lambda}$ and $M_{f,n}^{\lambda}$ depend on the energy (although the original continuous observable $f$ does not !). In fact, as one sees from \eqref{e:vari}, each eigenfunction $\mathring{\psi}_n$ comes with its own observable, generically denoted $K^{\lambda_n}$. This difficulty can be avoided in the special case of Theorem~\ref{cor:equireg}, where the observable $f$ is constant on each edge. See Lemma~\ref{lem:reduction} and the comment after Remark~\ref{rem:Weyl}. In fact, the analysis greatly simplifies in this case; one may bound the variance from above by a Hilbert-Schmidt norm without much difficulty (Lemma~\ref{lem:HSfacile}), then proceed to control this norm using the tools of quantum ergodicity for regular combinatorial graphs developed in \cite{ALM,A}. There is no need to speak about Green's functions at this point.

The general case is more delicate. To handle the energy-dependent observables $K^{\lambda_n}$, we use the non-backtracking eigenfunctions $f_j$, $f_j^\ast$ from \eqref{e:fjfjtoile}. We define the \emph{non-backtracking quantum variance} by
\begin{equation}\label{e:varnbi}
\varnbi(K^{\lambda}) = \frac{1}{N(I)}\sum_{\lambda_j\in I} \left|\langle f_j^{\ast}, K_B^{\lambda_j} f_j\rangle\right|^2
\end{equation}
for $K^{\lambda_j}\in \mathscr{H}_k$, where we denoted
\begin{equation}\label{e:kb}
\langle f, K_B g \rangle =\sum_{(x_0;x_k)\in B_k} \overline{f(x_0,x_1)}K(x_0;x_k)g(x_{k-1},x_k) \,.
\end{equation}

We show in \S~\ref{sec:discnonba} how to reduce the control of $\vari(K^{\lambda})$ to a control of $\varnbi(\tilde{K}^{\lambda})$. The result appears as \eqref{e:suff2}. This reduction is not new, so we briefly summarize it for the reader's convenience. Our next aim in \S~\ref{sec:vartonorm} is to give an upper bound for these variances in terms of a Hilbert-Schmidt norm, see Lemma~\ref{lem:HSgen}. This result is quite technical. One cannot apply the results of \cite{A,AS2} directly to the operators $J^m$ appearing in the proof. Still, we follow the general arguments of \cite[Section 4]{AS2}, and we were able to simplify many ideas for the specific framework of our paper. Now that we finally have Hilbert-Schmidt norms of discrete observables, we can follow the scheme of \cite{A} to conclude the proof.

So the case of general observables is more difficult, but worth the effort. As discussed in \S~\ref{sec:consequences}, it is this result that allows us to understand the behaviour of $|\psi_n(x)|^2$ \emph{inside} the edges. Our result says that the limiting measure will not be the Lebesgue measure in general, but will have a positive density $\Psi_{\lambda_n}(x)$ instead. A first description of $\Psi_{\lambda}$ appears in \eqref{e:psibad} while following the reduction. This definition, involving the trigonometric function $S_{\lambda}(x)$ introduced in \eqref{sec:DefQG}, doesn't really give any insight on the reason why the probability densities $|\psi_n(x)|^2$ approach this limit. This is why we spend certain effort to relate this to the Green function of the infinite quantum tree, by investigating further the spectral analysis of $H_{\mathbf{T}_q}$ initiated in \cite{Car97}. The result appears as \eqref{eq:2}, and finally gives the interesting interpretation that $|\psi_n(x)|^2$ will vary (in $x$) like the relative spectral density of $H_{\mathbf{T}_q}$, which is encoded by a quotient of Green's functions. We believe this is a quite important part of the paper, but the computations are heavy, so we collected them in an appendix. 

\subsection{Reduction to a discrete quantum variance} \label{sec:discrired}
We begin by showing that that $\|\psi_n\|_{L^2(\mathbf{G}_N)}$ can be expressed in terms of $\|\mathring{\psi}_n\|_{\ell^2(V_N)}$.

\begin{lem}
Let $(\psi_n)$ be an orthonormal basis of eigenfunctions with corresponding eigenvalues $(\lambda_n)$. For any $n$ such that $\lambda_n\notin \sigma_2$, we have
\begin{multline}\label{e:norsi}
\|\psi_n\|^2 = \frac{q+1}{s^2(\lambda_n)}\left(\int_0^L S^2_{\lambda_n}(x)\,\dd x\right)\|\mathring{\psi}_n\|^2 \\
 + \frac{1}{s^2(\lambda_n)}\left( \int_0^{L} S_{\lambda_n}(L-x)S_{\lambda_n}(x)\,\dd x\right)w(\lambda_n)\|\mathring{\psi}_n\|^2\,.
\end{multline}
\end{lem}
\begin{proof}
We may write
$\psi_n(x_b) = \psi_n(o_b) C_{\lambda_n}(x_b) + \psi_n'(o_b) S_{\lambda_n}(x_b)$,
expressing $\psi_n$ in the basis $C_{\lambda_n}(x)$ and $S_{\lambda_n}(x)$.
Since $\psi_n(t_b) = \psi_n(o_b) c(\lambda_n) + \psi_n'(o_b) s(\lambda_n)$, and since $\lambda_n\notin \sigma_2$ implies $s(\lambda_n)\neq 0$, we get using \eqref{e:simplif},
\begin{equation}\label{e:psinegal}
\psi_n(x_b) = \frac{S_{\lambda_n}(L-x_b)}{s(\lambda_n)} \psi_n(o_b) + \frac{S_{\lambda_n}(x_b)}{s(\lambda_n)}\psi_n(t_b)
\end{equation}

Thus,
\begin{align}
\int_0^{L} f_b(x)|\psi_n(x)|^2\,\dd x & = \frac{|\psi_n(o_b)|^2}{s^2(\lambda_n)} \int_0^{L} f_b(x)S^2_{\lambda_n}(L-x)\,\dd x \nonumber \\
&\quad + \frac{\overline{\psi_n(o_b)}\psi_n(t_b) + \psi_n(o_b)\overline{\psi_n(t_b)}}{s^2(\lambda_n)} \int_0^{L} f_b(x) S_{\lambda_n}(L-x)S_{\lambda_n}(x)\,\dd x 
\nonumber\\
&\quad + \frac{|\psi_n(t_b)|^2}{s^2(\lambda_n)}\int_0^{L} f_b(x)S^2_{\lambda_n}(x)\,\dd x \,. \label{eq:decompoSP}
\end{align}

Consider the special case where $f_b\equiv 1$ for all $b$. Since $\|\psi_n\|^2 = \frac{1}{2}\sum_{b\in B}\int_0^L |\psi_n(x)|^2\,\dd x$ (the half is due to directed edges), and since $\int_0^LS^2_{\lambda_n}(L-x)\,\dd x = \int_0^L S^2_{\lambda_n}(x)\,\dd x$, we get
\begin{align*}
2\,\|\psi_n\|^2 & = \frac{1}{s^2(\lambda_n)}\left( \int_0^L S^2_{\lambda_n}(x)\,\dd x\right)\left(\sum_{b\in B}|\psi_n(o_b)|^2 + \sum_{b\in B}|\psi_n(t_b)|^2\right)\\
& \quad + \frac{1}{s^2(\lambda_n)}\left( \int_0^{L} S_{\lambda_n}(L-x)S_{\lambda_n}(x)\,\dd x\right)\sum_{b\in B}(\overline{\psi_n(o_b)}\psi_n(t_b) + \psi_n(o_b)\overline{\psi_n(t_b)})\,.
\end{align*}
Now $\sum_b |\psi_n(o_b)|^2=\sum_b|\psi_n(t_b)|^2=(q+1)\|\mathring{\psi}_n\|^2$, $\sum_b \overline{\psi_n(o_b)}\psi_n(t_b) = \langle \mathring{\psi}_n,\cA_{G_N}\mathring{\psi}_n\rangle$ and $\sum_b\overline{\psi_n(t_b)}\psi_n(o_b) = \langle \mathcal{A}_{G_N}\mathring{\psi}_n,\mathring{\psi}_n\rangle$. 

The result then follows from \eqref{e:discreteeigen}, which tells us that $\cA_{G_N}\mathring{\psi}_n = w(\lambda_n)\mathring{\psi}_n$.
\end{proof}

Further to \eqref{e:vari} and \eqref{e:kg} we also introduce
\begin{equation}\label{e:avovk}
\langle K \rangle = \frac{1}{N} \sum_{(x_0;x_k)\in B_k} K(x_0;x_k) \quad \text{and} \quad \mathscr{H}_k^o = \{K\in \mathscr{H}_k : \langle K \rangle = 0 \}
\end{equation}
and notice that if $S_k\in \mathscr{H}_k$ are the constant functions 
\begin{equation}\label{eq:DefSk}
S_0(x_0) = 1 \qquad \text{and} \qquad S_k(x_0; x_k) = \frac1{(q+1)q^{k-1}}
\end{equation}
for $k\ge 1$, then for any $K\in \mathscr{H}_k$, we have $K-\langle K \rangle S_k \in \mathscr{H}_k^o$. 

The following lemma allows us to reduce the statement of Theorem \ref{thm:equireg} to a bound on the variances of three operators. Recall that $\kappa_{\lambda_n}$ and $\langle f\rangle_{\lambda_n}$ were defined in Theorem \ref{thm:equireg}.

\begin{lem}\label{lem:reduction}
Let $(\psi_n)$ be an orthonormal basis of eigenfunctions of $H_{\mathbf{G}_N}$ with eigenvalues $(\lambda_n)$. Assume $\lambda_n\in \overline{I} \subset  \sigma_{\mathrm{ac}}(H_{\mathbf{T}_q})$. Then:
\begin{enumerate}[\rm (i)]
\item We have
\begin{equation}\label{e:normpsidis}
\|\mathring{\psi}_n\|^2  =\frac{1}{\kappa_{\lambda_n}} \le C_{L,I}
\end{equation}
for some constant $C_{L,I}$;
\item
Given a function $f=(f_e)$ on the quantum graph $\mathbf{G}_N$, define $K_{f,n},J_{f,n}\in\mathscr{H}_0$ and $M_{f,n}\in \mathscr{H}_1$ by
\[
K_{f,n}(x_0) = \sum_{x_1\sim x_0} \frac{1}{s^2(\lambda_n)}\int_0^Lf_{(x_0,x_1)}(t)S^2_{\lambda_n}(L-t)\,\dd t \,,
\]
\[
J_{f,n}(x_1) = \sum_{x_0\sim x_1}  \frac{1}{s^2(\lambda_n)}\int_0^Lf_{(x_0,x_1)}(t)S^2_{\lambda_n}(t)\,\dd t \,,
\]
\[
M_{f,n}(x_0,x_1) = \frac{2}{s^2(\lambda_n)}\int_0^Lf_{(x_0,x_1)}(t)S_{\lambda_n}(L-t)S_{\lambda_n}(t)\,\dd t \,.
\]
Then
\begin{equation}\label{eq:ThreeTerms}
2 \,\langle \psi_n, f\psi_n\rangle = \langle \mathring{\psi}_n, (K_{f,n}+J_{f,n}+M_{f,n})_G\mathring{\psi}_n\rangle \,.
\end{equation}
Moreover, we have
\begin{multline}\label{e:withav}
2\left(\langle \psi_n, f\psi_n\rangle - \langle f\rangle_{\lambda_n} \right) = \langle \mathring{\psi}_n, (K_{f,n}-\langle K_{f,n}\rangle S_0)_G \mathring{\psi}_n\rangle + \langle \mathring{\psi}_n, (J_{f,n}-\langle J_{f,n} \rangle S_0) \mathring{\psi}_n\rangle \\+ \langle \mathring{\psi}_n, (M_{f,n}-\langle M_{f,n} \rangle S_1)_G \mathring{\psi}_n \rangle \,;
\end{multline}

\item
If $f$ is locally constant, i.e. $f_e(x) = c_e$ and $|c_e|\le 1$, define $K_f\in\mathscr{H}_0$ and $M_f\in\mathscr{H}_1$ by $K_f(x_0) = \sum_{x_1\sim x_0} c_{(x_0,x_1)}$ and $M_f(x_0,x_1) = c_{(x_0,x_1)}$. Then we get for $I\subset \sigma_{\mathrm{ac}}(H_{\mathbf{T}_q})$ and $\langle f\rangle$ as in Theorem~\ref{cor:equireg},
\[
\frac{1}{N(I)}\sum_{\lambda_n\in I} \left|\langle \psi_n,f\psi_n\rangle - \langle f\rangle\right|^2 \le C'_{L,I} \left(\vari(K_f-\langle K_f\rangle S_0) +\vari(M_f - \langle M_f\rangle S_1)\right) ,
\]
for some constant $C'_{L,I}$.
\end{enumerate}
\end{lem}
\begin{proof}
(i) The equality in (\ref{e:normpsidis}) follows from \eqref{e:norsi}, by the definition of $\kappa_{\lambda_n}$. To see the uniform bound, first note by Cauchy-Schwarz that $|\int_0^L S_{\lambda_n}(L-x)S_{\lambda_n}(x)\,\dd x| \le \int_0^L S^2_{\lambda_n}(x)\,\dd x$. We therefore have 
\[\kappa_{\lambda_n} \ge \frac{q+1 - |w(\lambda_n)|}{s^2(\lambda_n)}\int_0^LS^2_{\lambda_n}(x)\,\dd x.
\]
 As $\lambda_n\in \sigma_{\mathrm{ac}}(H_{\mathbf{T}_q})$, we have $|
w(\lambda_n)| \le 2\sqrt{q}$. Hence, $\kappa_{\lambda_n} \ge \frac{q+1-2\sqrt{q}}{s^2(\lambda_n)}\int_0^LS^2_{\lambda_n}(x)\,\dd x$. Finally, note that $I \ni \lambda \mapsto \frac{1}{s^2(\lambda)}\int_0^LS^2_{\lambda}(x)\,\dd x$ is continuous, so it reaches its minimum on $\overline{I} \subset \sigma_{\mathrm{ac}}(H_{\T_q})$, say at $\lambda_{\ast}$. This minimum cannot be zero, otherwise $S^2_{\lambda_{\ast}}(x)$ would be identically zero. Hence, $\frac{1}{s^2(\lambda_n)}\int_0^LS^2_{\lambda_n}(x)\,\dd x \ge c_{L,I}>0$, proving the claim.

(ii) Equation (\ref{eq:ThreeTerms}) readily follows from (\ref{eq:decompoSP}).
 To deduce \eqref{e:withav}, note that
\[
2\left(\langle \psi_n, f\psi_n\rangle - \langle f\rangle_{\lambda_n} \right) = \langle \mathring{\psi}_n, (K_{f,n}+J_{f,n}+M_{f,n})_G \mathring{\psi}_n\rangle - 2\,\langle f\rangle_{\lambda_n} \kappa_{\lambda_n}\cdot \|\mathring{\psi}_n\|^2\,.
\]
On the other hand,
\[
\langle \mathring{\psi}_n, (\langle K_{f,n}\rangle S_0 + \langle J_{f,n}\rangle S_0 + \langle M_{f,n}\rangle S_1)_G\mathring{\psi}_n\rangle = \Big(\langle K_{f,n} \rangle + \langle J_{f,n}\rangle + \frac{\langle M_{f,n}\rangle}{q+1} w(\lambda_n)\Big)\cdot\|\mathring{\psi}_n\|^2\,,
\]
where we used \eqref{e:discreteeigen}, that $\cA_{G_N} \mathring{\psi}_n = w(\lambda_n)\mathring{\psi}_n$. But

\begin{multline} \label{e:psibad}
\langle K_{f,n} \rangle + \langle J_{f,n}\rangle + \frac{\langle M_{f,n}\rangle}{q+1} w(\lambda_n) = \frac{1}{N} \sum_{(x_0,x_1)\in B} \frac{1}{s^2(\lambda_n)} \int_0^L f_{(x_0,x_1)}(t) \\ \cdot \left[S^2_{\lambda_n}(L-t)+S^2_{\lambda_n}(t) + \frac{2S_{\lambda_n}(L-t)S_{\lambda_n}(t)}{q+1}w(\lambda_n)\right]\dd t \\
=\frac{\kappa_{\lambda_n}}{N}\sum_{(x_0,x_1)\in B} \int_0^Lf_{(x_0,x_1)}(t)\Psi_{\lambda_n}(t)\,\dd t = 2\,\kappa_{\lambda_n}\langle f\rangle_{\lambda_n}\,,
\end{multline}
where we used \eqref{eq:2}. Hence, \eqref{e:withav} is proved by comparing the expressions.

(iii) If $f$ is locally constant, $f_{(x_0,x_1)}(t) = c_{(x_0,x_1)}$, we get $\langle f\rangle_{\lambda_n} = \langle f\rangle$ by \eqref{eq:3}. Also,
\[
\langle \mathring{\psi}_n,(K_{f,n})_G\mathring{\psi}_n\rangle = \langle \mathring{\psi}_n,(J_{f,n})_G\mathring{\psi}_n\rangle = a_{n,L} \langle \mathring{\psi}_n, (K_f)_G\mathring{\psi}_n\rangle \,,
\]
\[
\langle \mathring{\psi}_n,(M_{f,n})_G\mathring{\psi}_n\rangle = b_{n,L}\langle \mathring{\psi}_n, (M_f)_G\mathring{\psi}_n\rangle \,
\]
\[
\langle K_{f,n}\rangle = \langle J_{f,n}\rangle = a_{n,L} \langle K_f\rangle \quad \text{and} \quad \langle M_{f,n}\rangle = b_{n,L}\langle M_f\rangle \,.
\]
where $a_{n,L} =\frac{1}{s^2(\lambda_n)}\int_0^LS^2_{\lambda_n}(t)\,\dd t$ and $b_{n,L} = \frac{2}{s^2(\lambda_n)}\int_0^LS_{\lambda_n}(L-t)S_{\lambda_n}(t)\,\dd t$.
Since $|a_{n,L}| \le C_IL$ and $|b_{n,L}| \le C'_IL$ for $\lambda_n\in \overline{I}\subset \sigma_{\mathrm{ac}}(H_{\mathbf{T}_q})$, the proof is complete.
\end{proof}

Note that \eqref{e:withav} implies the following bound on the quantum variance~:
\begin{multline}
  \label{eq:simple_bound}
  \frac{1}{N(I)} \sum_{\lambda_n^{(N)}\in I} \left|\langle \psi_n^{(N)},
    f\psi_n^{(N)}\rangle -\langle f\rangle_{\lambda_n^{(N)}} \right|^2   
   \leq \\
  \vari(K_{f,n}^\lambda - \langle K_{f,n}^\lambda \rangle S_0)
  +\vari(J_{f,n}^\lambda - \langle J_{f,n}^\lambda \rangle S_0)
  +\vari(M_{f,n}^\lambda - \langle M_{f,n}^\lambda \rangle S_1).
\end{multline}

\subsection{From the variance to the non-backtracking variance} \label{sec:discnonba}

We now express \eqref{eq:simple_bound} in terms of non-backtracking
variances, following \cite{A}. Denote $\cN_x = \{y:y\sim x\}$.

As in \cite[\S 3]{A}, we introduce the stochastic operators $\cS:\mathscr{H}_k\to\mathscr{H}_k$,
\[ 
\cS :=\frac{\cA_{G}}{q+1} \quad \text{if } k=0 \qquad \text{and}\qquad \cS := \frac{1}{q}\cB  \quad \text{for } k\ge 1\,.
\]
Here, $(\cB f)(x_0;x_k)=\sum_{x_{k+1}\in \cN_{x_k}\setminus\{x_{k-1}\}} f(x_1,x_{k+1})$. We then define
\[
\cS_T := \frac{1}{T}\sum_{m=0}^{T-1}(T-m)\cS^m \qquad \text{and} \qquad \widetilde{\cS}_T  := \frac{1}{T}\sum_{m=1}^T \cS^m \,.
\]

Note that 
\[
K^{\lambda_j} = (I-\cS)(\cS_TK^{\lambda_j}) + \widetilde{\cS}_T K^{\lambda_j}.
\]

Let us define $i_k :\mathscr{H}_k \to \mathscr{H}_{k+1}$ by $(i_kJ)(x_0;x_{k+1}) = J(x_1;x_{k+1})$, $\nabla^{\ast}:\mathscr{H}_1\to\mathscr{H}_0$ by $(\nabla^{\ast}K)(x_0)=\sum_{x_{-1}\sim x_0}K(x_{-1},x_0)-\sum_{x_1\sim x_0}K(x_0,x_1)$ and $\nabla^{\ast}:\mathscr{H}_{k+1}\to \mathscr{H}_k$  by $(\nabla^{\ast} K)(x_0;x_k)=\sum_{x_{-1}\in \cN_{x_0}\setminus \{x_1\}} K(x_{-1};x_k)-\sum_{x_{k+1}\in\cN_{x_k}\setminus \{x_{k-1}\}} K(x_0;x_{k+1})$ for $k\ge 1$. Then we have 
\[
\nabla^{\ast} i_0 K = (q+1)(I-\cS) K \quad \text{and} \quad \nabla^{\ast}i_k K = q(I-\cS)K \quad \text{for} \ k\ge 1\,.
\]
Hence, using the fact that $\vari(K_1+K_2)\leq 2 \vari(K_1)+2\vari(K_2)$, we obtain
\begin{equation}\label{eq:FirstVarBound}
\vari(K^\lambda - \langle K^\lambda \rangle S_k) \le 2q^{-2} \vari(\nabla^{\ast}i_k \cS_TK^\lambda) + 2\vari(\widetilde{\cS}_TK^\lambda - \langle K^\lambda\rangle S_k\rangle\,,
\end{equation}
where $S_k$ is as in \eqref{eq:DefSk}.

Recalling \eqref{e:kg}, \eqref{e:kb} and \eqref{e:fjfjtoile}, a simple calculation shows that
\[
2\ii \Im \mu_j^-\cdot \langle \psi_j, (\nabla^{\ast} K)_G \psi_j \rangle = \langle f_j^{\ast}, K_B f_j \rangle - \langle g_j^{\ast},K_B g_j\rangle \,,
\]
where $g_j,g_j^{\ast}$ are defined like $f_j,f_j^{\ast}$, except that $\mu_j^-$ are replaced by $\overline{\mu_j^-}$. See also \cite[Lemma 6.4]{A} and \cite[Lemma 7.8]{A} for a similar argument.
It follows that\footnote{Actually, instead of $\varnbi(J)$, we should have $\frac{1}{2} \varnbi(J) + \frac{1}{2}\widetilde{\varnbi}(J)$, where $\widetilde{\varnbi}$ is defined using $g_j,g_j^{\ast}$ instead. Since this variance is controlled exactly like $\varnbi(J)$, except for replacing $\mu_j^-$ by $\overline{\mu_j^-}$ in the arguments, we omitted it for transparency.}
\[
\vari(\nabla^{\ast}i_k\cS_TK^{\lambda}) \le \max_j\frac{1}{|\Im \mu_j^-|^2} \varnbi(i_k\cS_TK^\lambda) \,.
\]
Assuming $\overline{I}\subset \sigma_{\mathrm{ac}}(H_{\mathbf{T}_q})$ is in a fixed band, so that $\overline{w(I)} \subset (-2\sqrt{q},2\sqrt{q})$, this maximum is bounded by some $C_I$. Back to \eqref{eq:FirstVarBound}, note that for any $K$, we have
\[
\widetilde{\cS}_T \langle K \rangle S_k = \langle K\rangle S_k \,,
\]
since $\cS \mathbf{1} = \mathbf{1}$. We thus finally get 
\begin{equation}\label{e:suff}
\vari(K^\lambda - \langle K^\lambda \rangle S_k) \le 2q^{-2}C_I \varnbi(i_k\cS_TK^\lambda) + 2\vari\left(\widetilde{\cS}_T[K^\lambda - \langle K^\lambda \rangle S_k]\right).
\end{equation}

\subsubsection*{Taking advantage of the invariance}
Let us introduce the operators $\cR_{n,r}^{\lambda_j}:\mathscr{H}_k\to\mathscr{H}_{n+k}$ defined by
\[
(\cR_{n,r}^{\lambda_j}K^{\lambda_j})(x_0;x_{n+k}) = \overline{(\mu_j^-)^{n-r}}(\mu_j^-)^r K^{\lambda_j}(x_{n-r};x_{n-r+k}) \,.
\]
Then $\langle f,(\cR_{n,r}^{\lambda_j}K)_B g\rangle = \langle (\cB^{\ast}\mu_j^-)^{n-r}f, K_B (\cB\mu_j^-)^r g \rangle$. 
It follows from \eqref{e:nonback} that $\varnbi(K^{\lambda}) = \varnbi(\frac{1}{n}\sum_{r=1}^n\cR_{n,r}^{\lambda}K^{\lambda})$. Using \eqref{eq:simple_bound} and \eqref{e:suff}, we thus obtain
\begin{equation}\label{e:suff2}
\begin{aligned}
\frac{1}{N(I)}\sum_{\lambda_j\in I} &\left|\langle \psi_j,f\psi_j\rangle - \langle f\rangle_{\lambda_j}\right|^2 \\
&\le C_I \varnbi\Big(\frac{1}{n}\sum_{r=1}^n\cR_{n,r}^{\lambda} i_0 \cS_TK_f^{\lambda} \Big) + C_I \varnbi\Big(\frac{1}{n}\sum_{r=1}^n\cR_{n,r}^{\lambda} i_0 \cS_TJ_f^{\lambda} \Big)\\
&\quad+ C_I \varnbi\Big(\frac{1}{n}\sum_{r=1}^n\cR_{n,r}^{\lambda} i_1 \cS_TM_f^{\lambda} \Big) + 2\vari\left(\widetilde{\cS}_T[K^\lambda_f - \langle K^\lambda_f \rangle S_0]\right)\\
&\quad+ 2\vari\left(\widetilde{\cS}_T[J^\lambda_f - \langle J^\lambda_f \rangle S_0]\right) + 2\vari\left(\widetilde{\cS}_T[M^\lambda_f - \langle M^\lambda_f \rangle S_1]\right)  \,,
\end{aligned}
\end{equation}
where $K_f^{\lambda_j} := K_{f,j}$, $J_f^{\lambda_j} := J_{f,j}$ and $M_f^{\lambda_j} := M_{f,j}$.

\subsection{Controlling the variances}\label{sec:vartonorm}

We next try to bound the variances in terms of Hilbert-Schmidt norms. For clarity, we start with energy-independent observables $K$.

\begin{lem}\label{lem:HSfacile}
Suppose that $I$ is in a fixed band of AC spectrum. Then for any $K\in\mathscr{H}_k$,
\[
\vari(K) \le C_{L,I} \frac{N}{N(I)} \|K\|_{\mathscr{H}}^2 + c_{k,q,L,I} \frac{N}{N(I)}\frac{\#\{x:\rho_{G_N}(x)<k\}}{N}\|K\|_{\infty}^2\,.
\]
where $\|K\|_{\mathscr{H}}^2 = \frac{1}{N}\sum_{(x_0;x_k)\in B_k}|K(x_0;x_k)|^2$ and $\|K\|_{\infty} = \max_{(x_0;x_k)} |K(x_0;x_k)|$.
\end{lem}
\begin{proof}
As we showed in Lemma \ref{lem:nonback}, if $\lambda$ is an eigenvalue of the quantum graph with eigenvector $\psi$, then $w(\lambda)$ is an eigenvalue of $\cA_{G_N}$ with eigenvector $\mathring{\psi}$. By assumption, if $\lambda_n,\lambda_j\in I$ and $\lambda_n\neq \lambda_j$, then $w(\lambda_n) \neq w(\lambda_j)$, so the eigenvectors $\mathring{\psi}_n$ and $\mathring{\psi}_j$ are orthogonal. If $\lambda_n = \lambda_j=\lambda$ and $\psi_n\perp \psi_j$, we find as in \eqref{e:norsi} that $\langle \mathring{\psi}_n, \mathring{\psi}_j\rangle \kappa_{\lambda} = 0$, so $\mathring{\psi}_n$ and $\mathring{\psi}_j$ are again orthogonal. If $\phi_n := \mathring{\psi}_n/\|\mathring{\psi}_n\|$, we may thus complete $(\phi_n)_{\lambda_n\in I}$ to an orthonormal basis $(\phi_n)_{n=1}^N$ of $\ell^2(G_N)$. Hence, recalling \eqref{e:vari} and using (\ref{e:normpsidis}), we get
\begin{multline*}
\vari(K) \le \frac{C_{L,I}}{N(I)} \sum_{\lambda_n\in I} \| K_G\mathring{\psi}_n\|^2 \le \frac{C_{L,I}^2}{N(I)}\sum_{\lambda_n\in I} \|K_G\phi_n\|^2 \\
\le \frac{C_{L,I}^2}{N(I)}\sum_{n=1}^N\|K_G\phi_n\|^2 = C_{L,I}^2\frac{N}{N(I)}\|K_G\|_{HSN}^2\,,
\end{multline*}
where $\|K_G\|_{HSN}^2 = \frac{1}{N}\sum_{x,y\in V_N} |K_G(x,y)|^2$ is a normalized HS norm. 
We may then apply \cite[Proposition 1]{A}, which tells us that 
\[
\|K_G\|_{HSN}^2\leq  \|K\|_{\mathscr{H}}^2 + c_{k,q,L,I} \frac{\#\{x:\rho_{G_N}(x)<k\}}{N}\|K\|_{\infty}^2
\]
for some $c_{k,q,L,I}>0$. This concludes the proof.
\end{proof}

\begin{rem}\label{rem:Weyl}
The fraction $\frac{N}{N(I)}$ arising in Lemma~\ref{lem:HSfacile} is asymptotically bounded. In fact, as $w$ is injective on $I$, we have $N(I) = \#\{w(\lambda_j) \in w(I)\}$. Moreover, $w(\lambda_j)$ are all eigenvalues of $\cA_{G_N}$. Conversely, if $\mu_j\in w(I)$, $\mu_j=w(t_j)$ is an eigenvalue of $\cA_{G_N}$, we may construct an eigenfunction of $H_{\mathbf{G}_N}$ with eigenvalue $t_j$ via formula \eqref{e:psinegal}. Hence, $N(I)$ is the number of eigenvalues of $\cA_{G_N}$ in $w(I)$, so $\frac{N(I)}{N} \to \mu_o(w(I)):=\langle \delta_o,\chi_{w(I)}(\cA_{\T_q})\delta_o\rangle$ by the law of Kesten-McKay \cite{Ke59,Mc81}. Since $w(I)\subset (-2\sqrt{q},2\sqrt{q})$, we have $\mu_o(w(I)) \ge C_I>0$. Hence\footnote{Note that this argument continues to hold on any $I$ on which $w$ is injective and $I\cap \sigma_2=\emptyset$. If moreover $I\cap \sigma_1=\emptyset$, we see that $\frac{N(I)}{N} \to \mu_o(w(I)) = 0$ because $w(I)\cap (-2\sqrt{q},2\sqrt{q})=\emptyset$.}, $\limsup_{N\to\infty} \frac{N}{N(I)} \le C_I^{-1}$.
\end{rem}

From here, one may adapt the ``ultra-short'' proof of \cite[\S 4]{A}, see also \cite[\S 2]{AS4}, to obtain a simple proof of Theorem~\ref{cor:equireg}. Indeed, in this case, the observables are independent of the energy (cf. Lemma~\ref{lem:reduction}.(iii)). Since our main concern is Theorem~\ref{thm:equireg}, we consider the general case directly. As the observables now depend on the energy (Lemma~\ref{lem:reduction}.(ii)), we use the non-backtracking quantum variance and give the following analogue of Lemma~\ref{lem:HSfacile}.

\begin{lem}\label{lem:HSgen}
Let $K^\lambda\in \mathscr{H}_k$, $k\ge 1$, depend on $\lambda$ as in \eqref{e:suff2}, and let $I$ be in a fixed band of the AC spectrum. Then
\begin{equation}\label{eq:BoundVarNB}
\limsup_{N\to\infty} \varnbi(K^\lambda) \le C(L,I) \limsup_{N\to\infty} \int_I \|K^{\lambda}\|_{\mathscr{H}}^2\,\dd\lambda \,,
\end{equation}
where $\|K^{\lambda}\|_{\mathscr{H}}^2 = \frac{1}{N}\sum_{(x_0;x_k)}|K^{\lambda}(x_0;x_k)|^2$. A similar bound holds for $\vari(K^{\lambda})$ if $k=0,1$.
\end{lem}
It is important that $C(L,I)$ is independent of $k$. In fact, this lemma will be applied to $K^{\lambda}\in \mathscr{H}_{n+k}$, with $n\to \infty$.
\begin{proof}
As in Lemma~\ref{lem:HSfacile}, if $\phi_j := \mathring{\psi}_j/\|\mathring{\psi}_j\|$, we may complete $(\phi_j)_{w(\lambda_j)\in w(I)}$ to an orthonormal basis $(\phi_j)_{j=1}^N$ of $\ell^2(G_N)$. Denote $\widetilde{f}_j = f_j/\|\mathring{\psi}_j\|$. We then have for $m_j=w(\lambda_j)$,
\[
\varnbi(K^{\lambda}) \le \frac{C_{L,I,q}}{N(I)} \sum_{m_j\in w(I)} \|K_B^{w^{-1}(m_j)} \widetilde{f}_j\|^2 \,.
\]

As previously mentioned, $m_j=w(\lambda_j)$ are eigenvalues of $\cA_{G_N}$ and we denote the rest of the eigenvalues corresponding to the completed base $(\phi_j)_{j=1}^N$ also by $\{m_j\}_{j=1}^N$. 

Since $w$ is injective and continuous on $I$, it is either strictly increasing or strictly decreasing. Assuming the former without loss of generality, we now define $J^{m_j} = K^{w^{-1}(m_j)}$ if $m_j\in w(I)$, $J^{m_j} = K^b$ if $m_j\ge w(b)$ and $J^{m_j} = K^a$ if $m_j \le w(a)$, where $[a,b] := \overline{I}$. 

Let $\chi$ be a continuous function, with $0\le \chi \le 1$, $\chi=1$ on $w(I)$, $\chi = 0$ outside a $\delta$-enlargement $w(I)_{\delta}$ of $w(I)$. We have
\begin{multline*}
\varnbi(K^{\lambda}) \le \frac{C_{L,I,q}}{N(I)} \sum_{m_j\in w(I)} \|J_B^{m_j} \widetilde{f}_j\|^2 \le  \frac{C_{L,I,q}}{N(I)} \sum_{j=1}^N \chi(m_j) \|J_B^{m_j} \widetilde{f}_j\|^2 \\
= \frac{C}{N(I)} \sum_{j=1}^N \chi(m_j) \sum_{(x_0,x_1)\in B}\sum_{(x_2;x_k),(y_2;y_k)} J^{m_j}(x_0;x_k)\overline{J^{m_j}(x_0;y_k)} \widetilde{f}_j(x_{k-1},x_k)\overline{\widetilde{f}_j(y_{k-1},y_k)}\,,
\end{multline*}
where the last sum runs over paths $(x_2;x_k)$, $(y_2;y_k)$ with $x_2,y_2\in \cN_{x_1}\setminus \{x_0\}$, and $(x_0;y_k) := (x_0,x_1,y_2,\dots,y_k)$. To keep the notations clear, we now assume that
\begin{equation}\label{e:kernelass}
J^{m}(x_0;x_k) = F(m) K(x_0;x_k)
\end{equation}
for some continuous function $F$ on $\R$, $|F(m)| \le c_I$. We comment on the case of the operators in \eqref{e:suff2} at the end of the proof.

Recalling \eqref{e:fjfjtoile} and the formula $f(\cA_G)(x,y) = \sum_{j=1}^N f(m_j) \phi_j(x)\overline{\phi_j(y)}$, we get
\begin{equation*}
\begin{aligned}
\varnbi(K^{\lambda}) &\leq \frac{C}{N(I)} \sum_{j=1}^N |F(m_j)|^2\chi(m_j) \sum_{(x_0,x_1)\in B}\sum_{(x_2;x_k),(y_2;y_k)} K(x_0;x_k)\overline{K(x_0;y_k)}\\
&\quad\times \Big( \big(\phi_j(x_k) - \nu(m_j)\phi_j(x_{k-1}) \big) \overline{\big( \phi_j(y_k) - \nu(m_j)\phi_j(y_{k-1})\big)}  \Big)
\\
&= \frac{C}{N(I)} \sum_{(x_0,x_1)\in B} \sum_{(x_2;x_k),(y_2;y_k)} K(x_0;x_k)\overline{K(x_0;y_k)}\big[(\chi |F|^2)(\cA_G)(x_k,y_k) \\
&\quad- (\overline{\nu}\chi |F|^2)(\cA_G)(x_k,y_{k-1}) - (\nu\chi |F|^2)(\cA_G)(x_{k-1},y_k) \\
&\quad+ \frac{1}{q} (\chi |F|^2)(\cA_G)(x_{k-1},y_{k-1})\big]\,,
\end{aligned}
\end{equation*}
where $\nu(m) = \mu^-(w^{-1}(m))$, and we used the fact that $|\mu^-(\lambda)|^2 = \frac{1}{q}$ for $\lambda\in I_{\delta}$.

As in \cite[Proposition 3]{A}, see also \cite[Theorem 4.1]{AS2}, we approximate each continuous function $\chi(m) |F(m)|^2,$ $\nu(m)\chi(m)|F(m)|^2$, $\overline{\nu}(m)\chi(m)|F(m)|^2$ by polynomials $Q_1,Q_2,Q_3$ up to an $\eps$ error. Let $d_{\eps} = \max(\deg(Q_i))$. 

Recall that $\rho_{G_N}(x)$ is the largest $\rho$ such that $B_{G_N}(x,\rho)$ has no closed cycles.
If $\rho_G(x_0) \ge d_{\eps}+k=:d_{\eps,k}$, the balls $B_{G_N}(x_0,d_{\eps,k})$ and $B_{\T_q}(\tilde{x}_0,d_{\eps,k})$ are thus isomorphic, so we get $Q_i(\cA_G)(x,y) = Q_i(\cA_{\T_q})(\tilde{x},\tilde{y})$ for all $x,y\in B_{G_N}(x_0,k)$, where $\tilde{x},\tilde{y}\in B_{\T_q}(\tilde{x}_0,k)$ are lifts of $x,y$. The ``bad'' $x_0$ with $\rho_G(x_0)< d_{\eps,k}$ induce an error term. Finally, we replace the polynomials $Q_i$ by the original functions, yielding an $\eps$ error. In the end, we get
\begin{multline*}
\varnbi(K^{\lambda}) \le \frac{C^2}{N(I)} \sum_{\rho_G(x_0) \ge d_{\eps,k}}\sum_{x_1\sim x_0} \sum_{(x_2;x_k),(y_2;y_k)} K(x_0;x_k)\overline{K(x_0;y_k)}\\
\big[(\chi |F|^2)(\cA_{\T_q})(\tilde{x}_k,\tilde{y}_k) - (\overline{\nu}\chi |F|^2)(\cA_{\T_q})(\tilde{x}_k,\tilde{y}_{k-1}) - (\nu\chi |F|^2)(\cA_{\T_q})(\tilde{x}_{k-1},\tilde{y}_k) \\
+ \frac{1}{q} (\chi|F|^2)(\cA_{\T_q})(\tilde{x}_{k-1},\tilde{y}_{k-1})\big]+ \frac{C' N(q+1)q^{2k}}{N(I)} \eps \, \|K\|_{\infty}^2  \\
+ C'\frac{\# \{\rho_G(x)<d_{\eps,k}\}}{N(I)}(q+1)q^{2k}\|K\|_{\infty}^2 \,
\end{multline*}
for some $C'>0$.

Since $\cA_{\T_q}$ has purely absolutely continuous spectrum, it follows that $f(\cA_{\T_q})(v,w) = \frac{1}{\pi}\int f(m) \Im G^m(v,w)\,\dd m$, where $G^m(v,w) = \lim_{\eta\downarrow 0}(\cA_{\T_q} - (m+\ii\eta))^{-1}(v,w)$ (see for instance \cite[Lemma 3.6]{AS}). Hence,
\begin{align*}
\varnbi(K^{\lambda}) &\leq \frac{c}{N(I)} \sum_{\rho_G(x_0) \ge d_{\eps,k}}\sum_{x_1\sim x_0} \sum_{(x_2;x_k),(y_2;y_k)}\Big{(} K(x_0;x_k)\overline{K(x_0;y_k)}\\
 &\quad\int_{w(I)_{\delta}} \chi(m)|F(m)|^2 \Big[\Im G^m(\tilde{x}_k,\tilde{y}_k) - \overline{\mu^-(w^{-1}m)}\Im G^m(\tilde{x}_k;\tilde{y}_{k-1})  \\
&\quad- \mu^-(w^{-1}m)\Im G^m (\tilde{x}_{k-1},\tilde{y}_k) + \frac{1}{q}\Im G^m(\tilde{x}_{k-1},\tilde{y}_{k-1})\Big]\,\dd m\Big{)}\\
&\quad+ \frac{C' N(q+1)q^{2k}}{N(I)} \eps \, \|K\|_{\infty}^2  + C'\frac{\# \{\rho_G(x)<d_{\eps,k}\}}{N(I)}(q+1)q^{2k}\|K\|_{\infty}^2.
\end{align*}

If $d(v,w)=s$, then $\Im G^m(v,w) = c_m \Phi_m(s)$, where $\Phi_m$ is the spherical function defined in (\ref{eq:DefSpherical}) and $c_m = \Im G^m(o,o)$. Let $e_k=(x_{k-1},x_k)$, $e_k'=(y_{k-1},y_k)$ and suppose $e_k\neq e_k'$. Then there is a path $(v_0,\dots,v_s)$ with $v_0=\tilde{x}_k$, $v_1=\tilde{x}_{k-1}$, $v_{s-1}=\tilde{y}_{k-1}$, $v_s=\tilde{y}_k$. So the term in square brackets takes the form
\begin{equation}\label{e:rec2}
c_m\left(\Phi_m(s) - \left[\overline{\mu^-(w^{-1}m)}+\mu^-(w^{-1}m)\right]\Phi_m(s-1) + \frac{1}{q}\Phi_m(s-2)\right) =0\,,
\end{equation}
where we used the fact that $q(\overline{\mu^-(t)}+\mu^-(t)) = w(t)$, as well as (\ref{eq:RecursSpherical}). Since $\rho_G(x_0)\ge d_{\eps,k}$, $e_k\neq e_k'$ iff $(y_2;y_k)\neq (x_2;x_k)$. We thus see that the sum vanishes if $(y_2;y_k)\neq (x_2;x_k)$, so only the diagonal terms may be nonzero. Recalling \eqref{e:inisphe}, we get
\begin{equation}\label{eq:BoundNBVar}
\begin{aligned}
&\varnbi(K^{\lambda}) \leq \frac{c}{N(I)} \sum_{\rho_G(x_0) \ge d_{\eps,k}} \sum_{(x_1;x_k)}\\
 &\int_{w(I)_{\delta}} \Big( |F(m)|^2|K(x_0;x_k)|^2\cdot \chi(m)c_m \Big[\frac{q+1}{q} -\frac{m^2}{q(q+1)}  \Big]\,\dd m\Big)\\
&+ \frac{C' N(q+1)q^{2k}}{N(I)} \eps \, \|K\|_{\infty}^2  + C'\frac{\# \{\rho_G(x)<d_{\eps,k}\}}{N(I)}(q+1)q^{2k}\|K\|_{\infty}^2 \,.
\end{aligned}
\end{equation}
Recall that $F(m)K(x_0;x_k) = J^m(x_0;x_k) = K^{w^{-1}(m)}(x_0;x_k)$. We finally change the variables $m = w(\lambda)$. Note that $|w'(\lambda)|\le C$ (in fact, $\lambda \mapsto w(\lambda)$ is analytic \cite[p.\ 10]{PT87}) and $c_m\le c_I$. Using Remark~\ref{rem:Weyl}, we see that the first term of (\ref{eq:BoundNBVar}) is asymptotically bounded by
\[  C(L,I) \limsup_{N\to\infty} \int_{I_{\delta'}} \|K^{\lambda}\|_{\mathscr{H}}^2\,\dd\lambda\]
where $\delta'$ goes to zero with $\delta$. The last term in (\ref{eq:BoundNBVar}) goes to zero as $N\rightarrow \infty$ thanks to \textbf{(BST)}. Therefore,  the proof is now complete for $\varnbi(K^{\lambda})$ by first taking $N\to\infty$, then $\eps\downarrow 0$ and $\delta \downarrow 0$.

The control of $\vari(K^{\lambda})$ for $k=0$ is similar. For $k=1$, we find that
\[
\vari(K^{\lambda}) \leq \frac{C}{N(I)} \int_{w(I)} \chi(m)\sum_{x_0\in V} \sum_{x_1,y_1\in\cN_{x_0}} J^{m}(x_0,x_1)\overline{J^{m}(x_0,y_1)}\Im G^m(\tilde{x}_1,\tilde{y}_1)\,\dd m\,.
\]
We write the sums $\sum_{x_0\in V}\sum_{x_1,y_1\in\cN_{x_0}}$ in the form
\begin{multline*}
\sum_{(x_0,x_1)} |J^{m}(x_0,x_1)|^2\Im G^m(\tilde{x}_1,\tilde{x}_1) + \sum_{(x_{-1},x_0,x_1)} J^{m}(x_0,x_1)\overline{J^{m}(x_0,x_{-1})}\Im G^m(\tilde{x}_1,\tilde{x}_{-1})\,.
\end{multline*}
The first term has the required form, since $\Im G^m(v,v)$ is bounded. For the second term, we have $\Im G^m(\tilde{x}_1,\tilde{x}_{-1}) = c_m \Phi_m(2)$. Using the Cauchy-Schwarz inequality, we have $|\sum_{(x_{-1},x_0,x_1)} J^{m}(x_0,x_1)\overline{J^{m}(x_0,x_{-1})} | \le q \sum_{(x_0,x_1)} |J^{m}(x_0,x_1)|^2$. This completes the proof.

Finally, in the previous proof we assumed \eqref{e:kernelass} that $J^m(x_0;x_k) = F(m)K(x_0;x_k)$. In \eqref{e:suff2}, for $k=1,2$, we have operators of the form
\[
J^m(x_0;x_{n+k}) = \frac{1}{n}\sum_{r=1}^n\sum_{s=0}^{T-1}\sum_{y_1,\dots,y_{s+1}\in B(x_{n-r+1},s+1)} \int_0^L F_{T,t,r,s}(m) f_{(y_s,y_{s+1})}(t)\,\dd t \,,
\]
where $F_{T,t,r,s}$ is continuous and uniformly bounded $|F_{T,t,r,s}(m)| \le C_I$. Looking back at the proof, we see that all arguments continue to apply; the notations simply become more cumbersome. Here, instead of $|F(m)|^2$, we have $F_{T,t,r,s}(m) \overline{F_{T',t',r',s'}(m)}$, which we approximate by a polynomial as before. 
\end{proof}

\subsection{End of the proof}\label{sec:end_proof}
We may now conclude the proof of Theorem~\ref{thm:equireg}. 
From (\ref{e:suff2}) and Lemma \ref{lem:HSgen}, we have
\begin{equation}\label{e:suff3}
\begin{aligned}
&\limsup_{N\to\infty} \frac{1}{N(I)}\sum_{\lambda_j\in I} \left|\langle \psi_j,f\psi_j\rangle - \langle f\rangle_{\lambda_j}\right|^2 \\
&\le C'(L,I) \limsup_{N\to\infty} \int_I \bigg[\left\|\frac{1}{n}\sum_{r=1}^n\cR_{n,r}^{\lambda} i_0 \cS_TK_f^{\lambda} \right\|_{\mathscr{H}}^2 + \left\|\frac{1}{n}\sum_{r=1}^n\cR_{n,r}^{\lambda} i_0 \cS_TJ_f^{\lambda}\right\|_{\mathscr{H}}^2\\
&\quad+ \left\|\frac{1}{n}\sum_{r=1}^n\cR_{n,r}^{\lambda} i_1 \cS_TM_f^{\lambda} \right\|_{\mathscr{H}}^2 + \left\|\widetilde{\cS}_T[K^\lambda_f - \langle K^\lambda_f \rangle S_0]\right\|_{\mathscr{H}}^2\\
&\quad+ \left\|\widetilde{\cS}_T[J^\lambda_f - \langle J^\lambda_f \rangle S_0]\right\|_{\mathscr{H}}^2 +\left\|\widetilde{\cS}_T[ M^\lambda_f - \langle M^\lambda_f \rangle S_1]\right\|_{\mathscr{H}}^2\bigg] \dd\lambda  \,,
\end{aligned}
\end{equation}

The first three terms may be controlled in exactly the same way as in \cite[p.\ 661--662]{A}, and give vanishing contributions as $N\rightarrow \infty$, followed by $n\to\infty$, for any $T$.

The remaining terms are of the form $\|\widetilde{\cS}_T J^{\lambda}\|_{\mathscr{H}}$, for $J^{\lambda} =K^{\lambda}-\langle K^{\lambda}\rangle S_k$. By \cite[Remark 2.2]{A}, we know that 
$\|\widetilde{\cS}_TJ^{\lambda}\|_{\mathscr{H}_k} \le \frac{c_{k,\beta}}{T} \|J^{\lambda}\|_{\mathscr{H}_k}$. The result follows by letting $T\rightarrow \infty$.

\section{The case of integral operators}\label{sec:integralop}

The aim of this Section is to prove Theorem~\ref{thm:integral}. We follow the same strategy discussed in Section~\ref{sec:equiregproof}. The main task will be to reduce the theorem to a control of discrete variances of observables in $\mathscr{H}_m^o$, see \eqref{eq:new_bound}. Afterwards, the proof will follow as before.

We start by expanding the scalar product \eqref{e:uneq}. Denote $b_1=(x_0,x_1)$ and $b_k=(x_{k-1},x_k)$. Using \eqref{e:psinegal} and the notation $S_n(x):=S_{\lambda_n}(x)$, we have
\begin{multline*}
2\left\langle \psi_n, K_k\psi_n\right\rangle \\= \sum_{(x_0;x_k)\in B_k} \overline{\psi_n(x_0)}\psi_n(x_{k-1})\int_0^L \int_0^L K_{b_1,b_k}(r_{b_1},s_{b_k})\frac{S_n(L-r_{b_1})S_n(L-s_{b_k})}{s^2(\lambda_n)}\,\dd r_{b_1}\,\dd s_{b_k}\\
+ \sum_{(x_0;x_k)\in B_k} \overline{\psi_n(x_1)}\psi_n(x_k)\int_0^L \int_0^L K_{b_1,b_k}(r_{b_1},s_{b_k})\frac{S_n(r_{b_1})S_n(s_{b_k})}{s^2(\lambda_n)}\,\dd r_{b_1}\,\dd s_{b_k}\\
+ \sum_{(x_0;x_k)\in B_k} \overline{\psi_n(x_0)}\psi_n(x_k)\int_0^L \int_0^L K_{b_1,b_k}(r_{b_1},s_{b_k})\frac{S_n(L-r_{b_1})S_n(s_{b_k})}{s^2(\lambda_n)}\,\dd r_{b_1}\,\dd s_{b_k}\\
+ \sum_{(x_0;x_k)\in B_k} \overline{\psi_n(x_1)}\psi_n(x_{k-1})\int_0^L \int_0^L K_{b_1,b_k}(r_{b_1},s_{b_k})\frac{S_n(r_{b_1})S_n(L-s_{b_k})}{s^2(\lambda_n)}\,\dd r_{b_1}\,\dd s_{b_k}\,.
\end{multline*}
Assume $k\ge 2$ (the case $k=1$ is similar to Theorem~\ref{thm:equireg}). If $J_{K,n} \in \mathscr{H}_k$, $M_{K,n}\in \mathscr{H}_{k-1}$ and $P_{K,n}\in \mathscr{H}_{k-2}$ are defined by
\[
J_{K,n}(x_0;x_k) = \int_0^L\int_0^L K_{(x_0;x_k)}(r_{b_1},s_{b_k})\frac{S_n(L-r_{b_1})S_n(s_{b_k})}{s^2(\lambda_n)}\,\dd r_{b_1}\,\dd s_{b_k}\,,
\]
\begin{multline*}
M_{K,n}(x_0;x_{k-1}) \\= \sum_{x_k\in\cN_{x_{k-1}}\setminus\{x_{k-2}\}} \int_0^L\int_0^L K_{(x_0;x_k)}(r_{b_1},s_{b_k})\frac{S_n(L-r_{b_1})S_n(L-s_{b_k})}{s^2(\lambda_n)}\,\dd r_{b_1}\dd s_{b_k}\\
+ \sum_{x_{-1}\in\cN_{x_0}\setminus\{x_1\}} \int_0^L\int_0^L K_{(x_{-1};x_{k-1})}(r_{b_0},s_{b_{k-1}})\frac{S_n(r_{b_0})S_n(s_{b_{k-1}})}{s^2(\lambda_n)}\,\dd r_{b_0}\,\dd s_{b_{k-1}}
\end{multline*}
\begin{multline*}
P_{K,n}(x_1;x_{k-1}) \\= \sum_{x_0\in \cN_{x_1}\setminus\{x_2\},x_k\in \cN_{x_{k-1}}\setminus\{x_{k-2}\}}\int_0^L\int_0^L K_{(x_0;x_k)}(r_{b_1},s_{b_k})\frac{S_n(r_{b_1})S_n(L-s_{b_k})}{s^2(\lambda_n)}\,\dd r_{b_1}\dd s_{b_k}\,,
\end{multline*}
where $K_{(x_0;x_k)} :=K_{(x_0,x_1),(x_{k-1},x_k)}$, we thus get
\begin{equation}\label{e:adecomposition}
2\langle \psi_n, K_k\psi_n\rangle = \langle \mathring{\psi}_n, (J_{K,n}+M_{K,n}+P_{K,n})_G \mathring{\psi}_n\rangle \,.
\end{equation}
On the other hand, for $\langle K\rangle$ and $S_k$ as in \eqref{e:avovk}, \eqref{eq:DefSk}, we have
\begin{equation}\label{e:anothereq}
\begin{aligned}
&\Big{\langle} \mathring{\psi}_n, \big{(}\langle J_{K,n}\rangle S_k + \langle M_{K,n}\rangle S_{k-1} + \langle P_{K,n}\rangle S_{k-2}\big{)}_G \mathring{\psi}_n \Big{\rangle} \\
&= \left[\langle J_{K,n}\rangle \Phi_{w(\lambda_n)}(k) + \langle M_{K,n}\rangle \Phi_{w(\lambda_n)}(k-1) + \langle P_{K,n}\rangle \Phi_{w(\lambda_n)}(k-2) \right] \cdot \|\mathring{\psi}_n\|^2 \\
&=: 2 \,\langle K_k\rangle_{n,L,\alpha}
\end{aligned}
\end{equation}
where in the first equality, we used \eqref{e:discreteeigen} and the fact that if $\tilde{\Phi}_k(\lambda) :=\Phi_{\lambda}(k)$, then\footnote{This is well-known~: by definitions \eqref{e:kg} and \eqref{eq:DefSk}, we have $[(S_k)_G\psi](x) = \frac{1}{(q+1)q^{k-1}}\sum_{(x_0;x_k),x_0=x} \psi(x_k)$ for $k\ge 1$ and $[(S_0)_G\psi](x)=\psi(x)$. Since $\tilde{\Phi}_0(\lambda) = 1$ and $\tilde{\Phi}_1(\lambda) = \frac{\lambda}{q+1}$, then \eqref{e:sphi} holds for $k=0,1$. Next, $[(S_{k+1})_G\psi](x) = \frac{1}{(q+1)q^k} \sum_{(x_0;x_{k+1}),x_0=x} \psi(x_{k+1}) = \frac{1}{(q+1)q^k}\sum_{(x_0;x_k),x_0=x}[(\cA_G\psi)(x_k)-\psi(x_{k-1})] = \frac{1}{q}[(\cS_k\cA_G\psi)(x) - (\cS_{k-1}\psi)(x)]$. If \eqref{e:sphi} holds for $j\le k$, then using \eqref{eq:RecursSpherical}, this yields $[\tilde{\Phi}_{k+1}(\cA_G)\psi](x)$, as asserted.}
\begin{equation}\label{e:sphi}
(S_k)_G = \tilde{\Phi}_k(\cA_G) \,.
\end{equation}

Using \eqref{e:normpsidis}, this yields the formula
\[
\langle K_k\rangle_{n,L,\alpha} = \frac{1}{N}\sum_{(b_1;b_k)\in B_k} \int_0^L \int_0^L K_{b_1,b_k}(x_{b_1},y_{b_k})\Psi_{\lambda_n,k}(x_{b_1},y_{b_k})\,\dd x_{b_1}\,\dd y_{b_k}
\]
where, for $k\ge 2$,
\begin{align}\label{e:psinlaphak}
\Psi_{\lambda_n,k}(x_{b_1},y_{b_k}) &= \frac{1}{2\,\kappa_{\lambda_n}}\bigg[\frac{S_n(L-x_{b_1})S_n(y_{b_k})}{s^2(\lambda_n)} \Phi_{w(\lambda_n)}(k) \\
& \quad + \frac{S_n(L-x_{b_1})S_n(L-y_{b_k}) + S_n(x_{b_1})S_n(y_{b_k})}{s^2(\lambda_n)}\Phi_{w(\lambda_n)}(k-1)\nonumber \\
& \quad + \frac{S_n(x_{b_1})S_n(L-y_{b_k})}{s^2(\lambda_n)}\Phi_{w(\lambda_n)}(k-2) \bigg]\,.\nonumber
\end{align}
For $k=1$, we have for $b=(x_0,x_1)$,
\begin{align}
\label{eq:psinlapha1}
\Psi_{\lambda_n,1}(x_b,y_b) & = \frac{1}{2\,\kappa_{\lambda_n}}\bigg[\frac{S_n(L-x_b)S_n(L-y_b) + S_n(x_b)S_n(y_b)}{s^2(\lambda_n)} \\
& \quad + \frac{S_n(L-x_b)S_n(y_b)+S_n(x_b)S_n(L-y_b)}{s^2(\lambda_n)}\Phi_{w(\lambda_n)}(1)\bigg]\,. \nonumber
\end{align}
These expressions are not very enlightening, so we now simplify them in terms of Green's functions. Using (\ref{eq:SphericalandResolvant}), we may express \eqref{e:psinlaphak} as

\begin{equation*}
\begin{aligned}
&\Psi_{\lambda_n,k}(x_{b_1},y_{b_k}) = \frac{1}{2\,\kappa_{\lambda_n}\Im G_{\mathbf{T}_q}^{\lambda_n}(o,o) s^2(\lambda_n) }\\
&\times \bigg[S_n(L-x_{b_1})S_n(y_{b_k}) \Im G^{\lambda_n}_{\mathbf{T}_q}(\tilde{o}_{b_1},\tilde{t}_{b_k}) 
+ S_n(L-x_{b_1})S_n(L-y_{b_k}) \Im G^{\lambda_n}_{\mathbf{T}_q}(\tilde{o}_{b_1},\tilde{o}_{b_{k}})\\
& + S_n(x_{b_1})S_n(y_{b_k})\Im G^{\lambda_n}_{\mathbf{T}_q}(\tilde{t}_{b_1},\tilde{t}_{b_k}) 
 + S_n(x_{b_1})S_n(L-y_{b_k}) \Im G^{\lambda_n}_{\mathbf{T}_q}(\tilde{t}_{b_1},\tilde{o}_{b_k}) \bigg]\,.
 \end{aligned}
\end{equation*}
Noting that $\tilde{y}_{b_k}\mapsto \Im G_{\mathbf{T}_q}^{\lambda}(v,\tilde{y}_{b_k})$ is an eigenfunction of $H_{\mathbf{T}_q}$, we see as in \eqref{e:psinegal} that
\[
\Im G^{\lambda_n}_{\mathbf{T}_q}(v,\tilde{y}_{b_k}) = \frac{S_n(L-y_{b_k})\Im G^{\lambda_n}_{\mathbf{T}_q}(v,\tilde{o}_{b_k}) + S_n(y_{b_k})\Im G^{\lambda_n}_{\mathbf{T}_q}(v,\tilde{t}_{b_k})}{s(\lambda_n)} \,,
\]
so we deduce that for $\tilde{x}_{b_1}\in \tilde{b}_1$,
\begin{equation}\label{e:psixy}
\begin{aligned}
&\Psi_{\lambda_n,k}(x_{b_1},y_{b_k}) \\
&= \frac{1}{2\kappa_{\lambda_n} \Im G_{\mathbf{T}_q}^{\lambda_n}(o,o)} \frac{S_n(L-x_{b_1})\Im G^{\lambda_n}_{\mathbf{T}_q}(\tilde{o}_{b_1},\tilde{y}_{b_k}) + S_n(x_{b_1})\Im G^{\lambda_n}_{\mathbf{T}_q}(\tilde{t}_{b_1},\tilde{y}_{b_k})}{s(\lambda_n)} \\
&= \frac{1}{2\kappa_{\lambda_n}}\cdot \frac{\Im G^{\lambda_n}_{\mathbf{T}_q}(\tilde{x}_{b_1},\tilde{y}_{b_k})}{\Im G_{\mathbf{T}_q}^{\lambda_n}(o,o)} \,,
\end{aligned}
\end{equation}
where we used that the first argument of $\Im G_{\mathbf{T}_q}^{\lambda}(\cdot,\cdot)$ is also an eigenfunction. This gives the expression in \eqref{e:psienfin}. Similarly, $\Psi_{\lambda_n,1}(x_b,y_b) =  \frac{1}{2\kappa_{\lambda_n}}\cdot \frac{\Im G^{\lambda_n}_{\mathbf{T}_q}(x_b,y_b)}{\Im G_{\mathbf{T}_q}^{\lambda_n}(o,o)}$. In other words, $\langle K_k\rangle_{n,L,\alpha} = \langle K_k\rangle_{\lambda_n}$. Recalling \eqref{e:adecomposition} and \eqref{e:anothereq}, we thus get
\begin{multline}
  \label{eq:new_bound}
\frac{1}{N(I)}\sum_{\lambda_n\in I} \left|\langle \psi_n, K_k\psi_n\rangle - \langle K_k\rangle_{\lambda_n}\right|^2 \le \vari(J_K^{\lambda}-\langle J_K^{\lambda} \rangle S_k) \\+ \vari(M_K^{\lambda}-\langle M_K^{\lambda}\rangle S_{k-1}) + \vari(P_K^{\lambda} - \langle P_K^{\lambda}\rangle S_{k-2})\,,
\end{multline}
where $J_K^{\lambda_n} := J_{K,n}$, $M_K^{\lambda_n} :=M_{K,n}$ and $P_K^{\lambda_n} :=P_{K,n}$.

\smallskip

From here, we argue as before. Namely, we reduce these discrete variances to non-backtracking ones, apply the invariance relation involving $\cR_{n,r}^{\lambda}$, and bound the variances by an HS norm. The only difference is that the variance of the remainders $\vari(\widetilde{\cS}_T[J - \langle J\rangle S_k])$ should now be controlled for general $k$ in Lemma~\ref{lem:HSgen}, not just $k=0,1$. Using the same ideas of the proof of this lemma, we arrive at a bound of the form
\[
\vari(K^{\lambda}) \lesssim \frac{c}{\pi N(I)} \int_{w(I)} \chi(m)\sum_{x_0\in V} \sum_{x_k,y_k\in S(x_0,k)} J^{m}(x_0;x_k)\overline{J^{m}(x_0;y_k)}\Im G^m(\tilde{x}_k,\tilde{y}_k)\,\dd m\,,
\]
where $S(x_0,k) = \{d(y,x_0)=k\}$. We may just bound this using $|\Im G^m(\tilde{x}_k,\tilde{y}_k)| \le c_k$, so the sum becomes $\sum_{x_0}|\sum_{x_k\in S(x_0,k)} |J^m(x_0;x_k)||^2 \le (q+1)q^{k-1} \sum_{(x_0;x_k)\in B_k} |J^m(x_0;x_k)|^2$. Of course, such a crude bound would be problematic if we needed to take $k\to \infty$ as in the control of the non-backtracking quantum variance (where $k$ is replaced by $n+k$, $n\to \infty$). In this case the expression should be bounded with more caution (in Lemma~\ref{lem:HSgen} for example, we used the recursive formula \eqref{e:rec2} to rule out off-diagonal terms, thereby avoiding Cauchy-Schwarz). Here however, $k$ is fixed and is the same as the one in $K_k$.

\bigskip

\noindent
{\bf{Acknowledgements:}} M.I.\ was funded by the LabEx IRMIA, and partially supported by the Agence Nationale de la Recherche project GeRaSic
(ANR-13-BS01-0007-01).\\
  M.S.\ was supported by a public grant as part of the \textit{Investissement 
d'avenir} project, reference ANR-11-LABX-0056-LMH, LabEx LMH.

We thank Nalini Anantharaman for suggesting this problem to us.

\appendix
\section{Some Green's function identities}\label{app:a}
\subsection{Relationship between resolvents on the quantum and combinatorial
trees}\label{app:a1}
Lemma~\ref{lem:Greenkernel} can be in principle derived from the results of 
\cite[Section 4]{Car97}.  Given that we use it in a crucial way, we
present here a self-contained proof.

The first step is to calculate $G^{\gamma}_{\mathbf{T}_q}(o_b,o_b)$ for some
bond $b\in \mathbf{T}_q$. Indeed by \eqref{eq:greensfunction}
we have 
\begin{equation*}
G^{\gamma}_{\mathbf{T}_q}(o_b,o_b) =  \frac{U_{\gamma;v}^-(o_b) V_{\gamma;o}^+(o_b)}{W^{\gamma}_{v,o}(o_b)}
\end{equation*}
for $\gamma\notin\sigma(H_{\mathbf{T}_q})$. If $b\in \cB^nb_o \cap \cB^{\ast m} b_v$, where $b_o$ and $b_v$ are the edges containing $o,v$ respectively, we get $U_{\gamma;v}^-(o_b) = \mu^-(\gamma)^mU_{\gamma}(0)$, $V_{\gamma;o}^+(o_b) = \mu^-(\gamma)^nV_{\gamma}(0)$ and $W^{\gamma}_{v,o}(o_b) = (\mu^-(\gamma))^{m+n} W^{\gamma}$, where $W^{\gamma}=V^{\gamma}(0)U_{\gamma}'(0) - V_{\gamma}'(0)U_{\gamma}(0) = s(\gamma)(1-(q\mu^+(\gamma))^2)$. Hence, 
\begin{multline}\label{eq:5}
G^{\gamma}_{\mathbf{T}_q}(o_b,o_b) = \frac{U_{\gamma}(0) V_{\gamma}(0)}{W^{\gamma}} = \frac{-s(\gamma) q\mu^+(\gamma)s(\gamma)}{s(\gamma)[1-(q\mu^+(\gamma))^2]} \\
= \frac{-s(\gamma)}{\mu^-(\gamma) - q\mu^+(\gamma)}  = \frac{-s(\gamma)}{(q+1)\mu^-(\gamma) - w(\gamma)}\,, 
\end{multline}
where we used \eqref{e:muplusmoins}. More generally, if $(v_0,\dots,v_k)$ is a non-backtracking path of vertices in $\T_q$, then we can take $o=v_0$, $v=v_k$ and
\begin{align*}
G^{\gamma}_{\mathbf{T}_q}(v_0,v_k) = \frac{U_{\gamma;v}^-(v_0) V_{\gamma;o}^+(v_k)}{W^{\gamma}_{v,o}(v_0)} &= (\mu^-(\gamma))^k \frac{U_{\gamma}(v_0)V_{\gamma}(v_0)}{W^{\gamma}_{v,o}(v_0)} = \frac{-s(\gamma)(\mu^-(\gamma))^k}{(q+1)\mu^-(\gamma) - w(\gamma)}
\end{align*} 
We thus showed that for any $\gamma\in \C^+$,
\begin{equation}\label{e:greenproof}
G_{\mathbf{T}_q}^{\gamma}(v,w) = -s(\gamma)\frac{(\mu^-(\gamma))^{d(v,w)}}{(q+1)\mu^-(\gamma)-w(\gamma)} = -s(\gamma)G_{\cA_{\T_q}}^{w(\gamma)}(v,w) \,,
\end{equation}
where the last equality is by the classic form of the Green's function of $\cA_{\T_q}$, see e.g. \cite[Chapter 1]{FTS} or \cite[Lemma 2.1]{AS}. Note that from \eqref{e:wlambda}, $\mu^{\pm}(\gamma)$ coincides with the quantity $\eps_{\pm}(w(\gamma))$ in \cite[Section 6]{A}, both being the solutions of $qx^2 - w(\gamma)x+1=0$.

If $\gamma = \lambda+i\eta$ with $\lambda\in \sigma_{\mathrm{ac}}(H_{\mathbf{T}_q})$, then taking $\eta \downarrow 0$ in \eqref{e:greenproof}, we get $G_{\mathbf{T}_q}^{\lambda+\ii 0}(v,w) = -s(\lambda)\frac{(\mu^-(\lambda))^{d(v,w)}}{(q+1)\mu^-(\lambda)-w(\lambda)}$. This coincides with $-s(\lambda)G_{\cA_{\T_q}}^{w(\lambda)+\ii 0}(v,w)$, proving the lemma\footnote{Note that $-s(\lambda)\frac{(\mu^-(\lambda))^{d(v,w)}}{(q+1)\mu^-(\lambda)-w(\lambda)}$ is well-defined for $\lambda \in \sigma_{\mathrm{ac}}(H_{\mathbf{T}_q})$. In fact, the denominator has a non-zero imaginary part in this case, so in particular it never vanishes. In general, it is easy to see that this fraction can have a pole only if $w(\lambda) = \pm (q+1)$. If $\alpha=0$, this only occurs for $\lambda\in \sigma_2$. Hence, $G_{\mathbf{T}_q}^{\lambda+\ii 0}(v,w)$ exists more generally for $\lambda\in \R\setminus \sigma_2$ in this case and equals $-s(\lambda)G^{w(\lambda)+\ii0}(v,w)$.}.

\subsection{Limiting measures and Green's functions}\label{app:a2}
If we re-write $U_\gamma(x)$ and $V_\gamma(x)$ in the basis $\{S_{\gamma}(x),S_{\gamma}(L-x)\}$, we find as in \eqref{e:psinegal},
\begin{displaymath}
  \begin{aligned}
    V_\gamma(x) &= q \mu^+(\gamma) S_{\gamma}(L-x) + S_{\gamma}(x) \,,\\
    U_\gamma(x) &= -\left( S_{\gamma}(L-x) + q \mu^+(\gamma) S_\gamma(x) \right),
  \end{aligned}
\end{displaymath}
If $x,y$ belong to the same bond $b=(o,v)$, and supposing that $y\in\mathbf{T}^+_x$, we have by \eqref{eq:greensfunction}, 
\begin{align*}
  G_{\mathbf{T}_q}^\gamma(x,y) &= \frac{\mu^-(\gamma)}{\mu^-(\gamma)} \cdot \frac{U_{\gamma}(x)V_{\gamma}(y)}{W^{\gamma}} \\
&=  - \frac{(\mu^-(\gamma) S_{\gamma}(L-x) + S_{\gamma}(x)) (q \mu^+(\gamma)S_{\gamma}(L-y) + S_{\gamma}(y))}{s(\gamma) [(q+1)\mu^-(\gamma)-w(\gamma)]}, 
\end{align*}
using \eqref{e:muplusmoins}. Taking $\gamma=\lambda+\ii 0$ with $\lambda\in \sigma_{\mathrm{ac}}(H_{\mathbf{T}_q})$ and recalling that $G^{\lambda}_{\mathbf{T}_q}(o,o) = \frac{-s(\lambda)}{(q+1)\mu^-(\lambda)-w(\lambda)}$ by \eqref{eq:5}, we get
\begin{multline*}
G_{\mathbf{T}_q}^\lambda(x,y) =  \frac{G_{\mathbf{T}_q}^\lambda(o,o)}{s^2(\lambda)} \Big( S_{\lambda}(L-x)S_{\lambda}(L-y) + S_{\lambda}(x)S_{\lambda}(y) \\
 + \mu^-(\lambda)S_{\lambda}(L-x)S_{\lambda}(y) + q\mu^+(\lambda)S_{\lambda}(x)S_{\lambda}(L-y) \Big)\,.
\end{multline*}
Hence,
\begin{multline*}
\Im G^{\lambda}_{\mathbf{T}_q}(x,y) = \frac{\Im G_{\mathbf{T}_q}^{\lambda}(o,o)}{s^2(\lambda)} \bigg(S_{\lambda}(L-x)S_{\lambda}(L-y) + S_{\lambda}(x)S_{\lambda}(y) \\
+ \frac{w(\lambda)}{2q}S_{\lambda}(L-x)S_{\lambda}(y) + \frac{w(\lambda)}{2}S_{\lambda}(x)S_{\lambda}(L-y) \bigg) \\
+ \frac{\Re G_{\mathbf{T}_q}^{\lambda}(o,o)}{s^2(\lambda)}\Big(\Im \mu^-(\lambda)S_{\lambda}(L-x)S_{\lambda}(y) + \Im q\mu^+(\lambda)S_{\lambda}(x)S_{\lambda}(L-y)\Big) \,.
\end{multline*}
Now $\Im G^{\lambda}_{\mathbf{T}_q}(o,o) = \frac{(q+1)\Im\mu^-(\lambda)s(\lambda)}{|(q+1)\mu^-(\lambda)-w(\lambda)|^2}$ and $\Re G^{\lambda}_{\mathbf{T}_q}(o,o) = \frac{q-1}{2q}\frac{w(\lambda)s(\lambda)}{|(q+1)\mu^-(\lambda)-w(\lambda)|^2}$. Hence, $\Re G^{\lambda}_{\mathbf{T}_q}(o,o) = \frac{q-1}{2q}\frac{w(\lambda)}{(q+1)\Im \mu^-(\lambda)} \Im G^{\lambda}_{\mathbf{T}_q}(o,o)$. We thus get
\begin{multline}
\label{eq:greenPsi}
\Im G_{\mathbf{T}_q}^{\lambda}(x,y) = \frac{\Im G_{\mathbf{T}_q}^{\lambda}(o,o)}{s(\lambda)^2}
\bigg( S_{\lambda}(L-x)S_{\lambda}(L-y) + S_{\lambda}(x)S_{\lambda}(y) \\
 + \frac{w(\lambda)}{2q}S_{\lambda}(L-x)S_{\lambda}(y) + \frac{w(\lambda)}{2} S_{\lambda}(x)S_{\lambda}(L-y)  \\
+ \frac{(q-1)w(\lambda)}{2q(q+1)}  S_{\lambda}(L-x)S_{\lambda}(y) - \frac{q-1}{2(q+1)}w(\lambda) S_{\lambda}(x)S_{\lambda}(L-y)  \bigg)\\
= \frac{\Im G_{\mathbf{T}_q}^{\lambda}(o,o)}{s(\lambda)^2} \bigg( S_{\lambda}(L-x)S_{\lambda}(L-y) + S_{\lambda}(x)S_{\lambda}(y)   \\
+ \frac{w(\lambda)}{q+1} \left[ S_{\lambda}(L-x)S_{\lambda}(y) + 
S_{\lambda}(x)S_{\lambda}(L-y)\right]\bigg), 
\end{multline}
which recovers $\Psi_{\lambda_n,1}(x,y) = \frac{1}{2\kappa_{\lambda_n}}\cdot 
\frac{\Im G^{\lambda_n}_{\mathbf{T}_q}(x,y)}
{\Im G_{\mathbf{T}_q}^{\lambda_n}(o,o)}$ from \eqref{eq:psinlapha1}.
Because the final expression in \eqref{eq:greenPsi} is symmetric in $x$ and
$y$, it holds for $y \in \mathbf{T}_x^-$ too.

By a similar method we may also recover $\Psi_{\lambda_n,k}(x_{b_1},y_{b_k})$ with the help of identity \eqref{eq:SphericalandResolvant}, thus obtaining an alternative proof of \eqref{e:psixy}.

If we specialize to $y=x$ in \eqref{eq:greenPsi}, we get
\begin{equation}\label{eq:2}
\frac{\Im G_{\mathbf{T}_q}^{\lambda+ \mathrm{i}0}(x,x)}{\Im G_{\mathbf{T}_q}^{\lambda+\mathrm{i}0}(o,o)} = \frac1{s(\lambda)^2} \left( S_{\lambda}^2(L-x) + S_{\lambda}^2(x) + \frac{2w(\lambda)}{q+1} S_{\lambda}(L-x)S_{\lambda}(x)\right).
\end{equation}

In particular, recalling \eqref{e:kappalambda}, we obtain
\begin{equation}\label{eq:3}
\int_0^L \Psi_{\lambda_n}(x)\,\dd x = \frac{1}{\kappa_{\lambda_n}} \int_0^L \frac{\Im G^{\lambda_n+\ii0}_{\mathbf{T}_q}(x,x)} {\Im G_{\mathbf{T}_q}^{\lambda_n+\ii0}(o,o)}\,\dd x = \frac{2}{q+1} \,.
\end{equation}

Note that, when $\lambda\in \sigma_1$, we have by (\ref{e:spechh}) that 
\begin{equation}\label{eq:2bis}
\frac{\Im G_{\mathbf{T}_q}^{\lambda+ \mathrm{i}0}(x,x)}{\Im G_{\mathbf{T}_q}^{\lambda+\mathrm{i}0}(o,o)} \geq \frac1{s(\lambda)^2} \left( S_{\lambda}^2(L-x) + S_{\lambda}^2(x) - \frac{4\sqrt{q}}{q+1} S_{\lambda}(L-x)S_{\lambda}(x)\right) > 0,
\end{equation}
since $\frac{4\sqrt{q}}{q+1} <2$ for any $q$.  The inequality in \eqref{eq:2bis}
is strict since $S_\lambda(L-x)$ and $S_\lambda(x)$ cannot simultaneously
be $0$ as they are independent solutions to a linear ODE.

If we consider the special case where $\alpha=0$, $U\equiv0$ we find the
interesting property that $\Im G_{\mathbf{T}_q}^{\lambda+ \mathrm{i}0}(x,x)$
is constant as a function of $x$.  Indeed we have then $S_{\lambda}(x)=\frac{\sin(
\sqrt{\lambda}x)}{\sqrt{\lambda}}$ and $w(\lambda)=(q+1)\cos(\sqrt{\lambda}L)$, so \eqref{eq:2} yields
\begin{align*}
\frac{\Im G_{\mathbf{T}_q}^{\lambda+ \mathrm{i}0}(x,x)}
  {\Im G_{\mathbf{T}_q}^{\lambda+\mathrm{i}0}(o,o)}
&= \frac1{\sin^2(\sqrt{\lambda}L)}
\Big( \left(\sin(\sqrt{\lambda}L) \cos(\sqrt{\lambda}x) -
 \sin(\sqrt{\lambda}x) \cos(\sqrt{\lambda}L) \right)^2 + \sin^2(\sqrt{\lambda}x)
\\&\quad 
+2\cos(\sqrt{\lambda}L) \left( \sin(\sqrt{\lambda}L) \cos(\sqrt{\lambda}x) -
 \sin(\sqrt{\lambda}x) \cos(\sqrt{\lambda}L)\right) \sin(\sqrt{\lambda}x)
\Big)\\
&= \frac1{\sin^2(\sqrt{\lambda}L)} \big( \sin^2(\sqrt{\lambda}L) 
\cos^2(\sqrt{\lambda}x) - \sin^2 (\sqrt{\lambda}x) \cos^2(\sqrt{\lambda}L)
+ \sin^2(\sqrt{\lambda}x)\big)\\
&= 1.
\end{align*}
Using \eqref{eq:3}, we get $\kappa_{\lambda_n} = \frac{q+1}2L$ in this case. Hence, we deduce that the limiting measure in the special case $\alpha=0$, $U\equiv0$
is the uniform measure~:
\begin{equation}
\label{eq:4}
\Psi_{\lambda_n}(x)
= \frac2{L(q+1)}.
\end{equation}


\providecommand{\bysame}{\leavevmode\hbox to3em{\hrulefill}\thinspace}
\providecommand{\MR}{\relax\ifhmode\unskip\space\fi MR }
\providecommand{\MRhref}[2]{%
  \href{http://www.ams.org/mathscinet-getitem?mr=#1}{#2}
}
\providecommand{\href}[2]{#2}

\end{document}